\documentclass[reqno]{amsart}

\usepackage{mathtools,amsmath,amssymb}
\usepackage[pdfencoding=auto, psdextra]{hyperref}
\hypersetup{colorlinks=true,allcolors=[rgb]{0,0,0.6}}

\allowdisplaybreaks

\usepackage[hmargin= 1.3 in,vmargin=1 in]{geometry}

\usepackage{enumerate}
\usepackage{hyperref}

\usepackage{amsmath}
\usepackage{amsfonts}
\usepackage{amssymb}
\usepackage{graphicx}
\usepackage{amsthm,color,yfonts,cite, latexsym}
\usepackage{paralist}
\usepackage{mathrsfs}
\usepackage{hyperref}
\usepackage{physics}
\usepackage{bbm}
\usepackage{bm}
\usepackage{mathtools}
\mathtoolsset{showonlyrefs}

\usepackage{tikz}
\usepackage{tikz-cd}
\usetikzlibrary{cd}
\usetikzlibrary{arrows.meta}
\usetikzlibrary{matrix, arrows, positioning}
\usetikzlibrary{shadows}

\usepackage{tikz}
\usepackage{tikz-cd}
\usetikzlibrary{cd}
\usetikzlibrary{arrows.meta}
\usetikzlibrary{matrix, arrows, positioning}
\usetikzlibrary{shadows}

\newtheorem{theorem}{Theorem}
\newtheorem{definition}[theorem]{Definition}
\newtheorem{proposition}[theorem]{Proposition}

\newtheorem{lemma}[theorem]{Lemma}
\newtheorem{corollary}[theorem]{Corollary}

\newtheorem{remark}[theorem]{Remark}

\makeatletter
\newcommand*{\rom}[1]{\expandafter\@slowromancap\romannumeral #1@}
\makeatother

\newcommand{\ls}{\lesssim}

\renewcommand{\lg}{\langle}
\newcommand{\rg}{\rangle}
\newcommand{\R}{\mathbb{R}}
\newcommand{\T}{\mathbb{T}}
\newcommand{\C}{\mathbb{C}}
\newcommand{\Z}{\mathbb{Z}}

\newcommand{\ep}{\varepsilon}
\newcommand{\al}{\alpha}

\newcommand{\BN}{\mathbb{N}}

\newcommand{\CA}{\mathcal{A}}
\newcommand{\CB}{\mathcal{B}}

\newcommand{\CH}{\mathcal{H}}

\newcommand{\CS}{\mathcal{S}}

\newcommand{\SA}{\mathscr{A}}
\newcommand{\SB}{\mathscr{B}}

\newcommand{\FS}{\mathfrak{S}}

\newcommand{\ga}{\gamma}
\newcommand{\de}{\delta}
\newcommand{\ps}{\psi}

\newcommand{\ta}{\tau}
\renewcommand{\th}{\theta}
\newcommand{\lm}{\lambda}
\newcommand{\si}{\sigma}
\newcommand{\om}{\omega}

\newcommand{\rh}{\rho}

\newcommand{\Ga}{\Gamma}
\newcommand{\De}{\Delta}
\newcommand{\Ph}{\Phi}
\newcommand{\Ps}{\Psi}

\newcommand{\Om}{\Omega}

\newcommand{\pl}{\partial}
\newcommand{\wt}{\widetilde}
\newcommand{\wh}{\widehat}

\newcommand{\Ck}[1]{\left\{#1\right\}}
\newcommand{\Dk}[1]{\left[#1\right]}
\newcommand{\K}[1]{\left(#1\right)}

\newcommand{\No}[1]{\left\| #1 \right\|}

\newcommand{\I}{\infty}
\newcommand{\sd}{\langle  \partial_x \rangle}

\newcommand{\tx}{{t,x}}

\newcommand{\tw}{\frac{1}{2}}

\newcommand{\II}{\mathbbm{1}}

\newcommand{\gs}{\gtrsim}

\newcommand{\ov}{\overline}

\renewcommand{\Tr}{\operatorname{Tr}}

\newcommand{\vrh}{\ov{\rh}}
\newcommand{\PN}{P_{\le N}}
\newcommand{\N}{\BN}

\usepackage{wrapfig}
\usepackage{tikz}
\usetikzlibrary{arrows,calc,decorations.pathreplacing}
\definecolor{light-gray1}{gray}{0.90}
\definecolor{light-gray2}{gray}{0.80}
\definecolor{light-gray3}{gray}{0.60}

\numberwithin{equation}{section}

\numberwithin{theorem}{section}

\numberwithin{table}{section}

\numberwithin{figure}{section}

\makeatletter
\def\l@section{\@tocline{1}{0pt}{0pt}{}{}}%

\def\l@subsection{\@tocline{2}{0pt}{2.5em}{}{}}%

\def\l@subsubsection{\@tocline{3}{0pt}{3.0em}{}{}}%

\makeatother

\title[Orthonormal Strichartz estimates for renormalised densities]{Applications of renormalisation to orthonormal Strichartz estimates and the NLS system on the circle}
\author[S. Hadama]{Sonae Hadama}
\address{The University of Osaka, 1-1 Machikaneyama-cho, Toyonaka-shi, Osaka 560-0043, Japan}
\email{hadama.sonae.sci@osaka-u.ac.jp}

\author[A. Rout]{Andrew Rout}
\address{Dipartimento di Matematica, Politecnico di Milano, P.zza Leonardo da Vinci 32, 20133, Milano, Italy.}
\email{andrewjames.rout@polimi.it}

\thanks{The first author was supported by JSPS KAKENHI Grant Number 24KJ1338. The second author acknowledges the support by the Italian Ministry of University and Research (MUR) through
the PRIN 2022 grant ``OpeN and Effective quantum Systems (ONES)''}
\date{\today}

\begin{document}
\begin{abstract}
{In this paper, we introduce a renormalisation procedure for the density associated with the system of nonlinear Schr\"odinger equations (NLSS) on a circle. We show that this renormalised density satisfies better orthonormal Strichartz estimates than the non-renormalised density, which was considered in Nakamura (2020). As an application, we determine the critical Schatten exponent below which the cubic renormalised NLSS on the circle is globally well-posed and above which it is ill-posed. Finally, we show that the improvement for orthonormal Strichartz estimates satisfied by the renormalised density on $\T^d$ for $d \ge 2$ is minimal.}
\end{abstract}
\maketitle
\tableofcontents
\section{Introduction}
\subsection{Nonlinear Schr\"odinger system}
\label{SEC:NLSS:1}
The cubic nonlinear Schr\"odinger system (NLSS) is given by
\begin{equation}
\label{NLSS_introduction}
\begin{dcases}
i \partial_t u_n = -\Delta u_n \pm \rho u_n, \quad n \in \mathbb{N}, \\
\rh = \sum_{m \in \BN} \lambda_m |u_m(t,x)|^2 \\
u_{n}(0) = \phi_n \in H^{s}(X),
\end{dcases}
\end{equation}
where $(\lambda_m) \in \ell^\alpha(\N)$. If one adds the assumption that the $(\phi_n)_{n \in \mathbb{N}}$ are an orthonormal system, then one can use the NLSS as a model for a mixture of fermions, see the introduction of \cite{LS15}. In the case of an orthonormal system, a natural question is the optimal choices of $\al$ and $s$ for which there is a local solution to the NLSS. This question is usually answered by invoking so-called {\it orthonormal Strichartz estimates}. In this paper, we introduce a renormalisation procedure for the density. More precisely, one can rewrite \eqref{NLSS_introduction} as
\begin{equation}
\label{cubic_Hartree_introduction}
\begin{cases}
i \partial_t \gamma = [-\Delta \pm \rho_\gamma, \gamma], \\
\gamma(t=0) = \gamma_0,
\end{cases}
\end{equation}
where $\gamma$ denotes the one-particle density operator of the system. For a more precise definition of the objects in \eqref{cubic_Hartree_introduction}, we direct the reader to Section \ref{SEC:NLSS:introduction}. Our renormalisation procedure is motivated by the simple observation that for any finite constant $c \in \C$, one has that
\begin{equation*}
[A + c, B] = [A,B].
\end{equation*}
In particular, one can formally subtract a constant from the density without changing the NLSS. We will thus define a {\it renormalised density} $\overline{\rho}_A$ which is formally defined by  
\begin{equation*}
\vrh_A(x) := \rh_{A}(x) - \frac{1}{ \mathrm{Vol} (X)}\Tr A  \equiv \rh_A(x) - \frac{1}{\mathrm{Vol}(X)}\int_{X} d x' \rh_A(x')
\end{equation*}
when $\mathrm{Vol}(X)<\I$.
For trace-class operators, this is a finite constant. However, for non trace-class operators, we will need to be more careful, see Remark \ref{remark_non_trace}. We will show that when working on the one dimensional torus, this renormalised density satisfies better (orthonormal) Strichartz estimates compared to the standard density, see Theorems \ref{thm:L2 bound renormalised} and \ref{thm:L3 bound renormalised}.
\subsection{Notation}
Let us fix the notation that we will use throughout the paper. We will use $C$ to denote a generic constant, which may change line-to-line. When $C$ depends on the variables $x_i$, we will write $C = C(x_i)$. We also adopt the notation $A \lesssim B$ to denote $A \leq C B$, and similarly $A \lesssim_{x_i} B$. We will write $A \sim B$ when $A \lesssim B$ and $B \lesssim A$.
\subsubsection*{Fourier coefficients}
We fix the spatial domain $X := \T:=\R/2\pi\Z$ and the basis functions
\begin{equation}
\label{en_defintion}
e_n(x) := \frac{1}{\sqrt{2\pi}} e^{inx}.
\end{equation}
Then for a function $\phi \in L^2(\T,\C) \equiv L^2(\T)$, we define its Fourier coefficients
\begin{equation*}
\wh{\phi}(n):= \int_\T dx \, \ov{e_{n}(x)}\phi(x).
\end{equation*}
We also consider the Dirichlet projection operator $P_{\le N}$ defined via
\begin{equation*}
P_{\le N} \phi := \sum_{n=-N}^N \wh{\phi}(n) \, e_n(x).
\end{equation*}
We also adopt the notation that for $f:[0,T]\times \T_x \to \C$, the average of $f$ is given by
	\begin{equation*}
		\lg f \rg := \frac{1}{2\pi T} \int_0^T \int_{\T} dx \, dt \, f (t,x), 
        \quad \lg f(t) \rg_x := \frac{1}{2\pi}\int_\T dx \, f(t,x).
	\end{equation*}
\subsubsection*{Free Schr\"odinger kernel}
The free Schr\"odinger equation on $\T$ is given by
\begin{equation}
\label{free_Schro_T}
	\begin{dcases}
		i\pl_t u + \De u = 0, & \\
		u(0) = \phi \in L^2(\T).
	\end{dcases}
\end{equation}
We denote by $U(t)$ the {\it free Schr\"odinger kernel} given by
\begin{equation}
\label{free_schro_kernel_definition}
U(t) \phi \equiv e^{it\De}\phi(x) := \sum_{n\in\Z} e_n(x)e^{-itn^2}\wh{\phi}(n).
\end{equation}
We will write $U^*(t):= U(t)^*$. More generally, we write $(A(t))^*:=A^*(t)$ for any family of linear operators on $L^2(\T)$.

\subsubsection*{Schatten classes and operators}
Let us recall some standard terminology from the operator theory.
Let $H$ and $K$ be complex and separable Hilbert spaces. Let $A:H \to K$ be a linear operator. 
We write $\CB(H\to K)$ to denote the space of all bounded operators from $H$ to $K$, and write $\CB(H) \equiv \CB(H \to H)$.
We define $|A|:= \sqrt{A^* A}$.
Moreover, we define the Schatten--$\al$ class $\FS^\al(H\to K)$ by
\begin{equation}
	\|A\|_{\FS^\al(H\to K)} := \begin{dcases}
		(\Tr_H(|A|^\al))^{1/\al}\quad &\text{if } 1\le \al < \I,\\
        \sup \mathrm{spec} |A| \quad &\text{if } \al = \I,
	\end{dcases}
\end{equation}
where we denote the spectrum of an operator by $\mathrm{spec}$. Then $\FS^1(H\to K)$ is the set of trace-class operators and $\FS^2(H\to K)$ is the set of Hilbert--Schmidt operators.
In general, we have $1\le \al_1 < \al_2 \le \I$ implies $\FS^{\al_1}(H\to K)\subsetneq \FS^{\al_2}(H\to K)$.
When $A$ is compact we can use singular value decomposition to write
\begin{equation}\label{eq:spectral decomposition}
	A = \sum_{n=1}^\I a_n|\phi_n \rg \lg \ps_n|,
\end{equation}
where $a_n \ge a_{n+1}\ge 0$ and $(\phi_n)_{n=1}^\I \subset K$, $(\ps_n)_{n=1}^\I \subset H$ are orthonormal systems. At times, it will also be helpful to identify integral operators $\gamma_0$ with their integral kernels $\gamma_0(x,x')$.
\subsection{Orthonormal Strichartz estimates}
For the convenience of the reader, we briefly recall some background on orthonormal Strichartz estimates when $X= \R^d$. The free Schr\"{o}dinger equation on $\R^d$ is given by
\begin{equation*}
	\begin{dcases}
		i\pl_t u + \De_{\R^d} u = 0, \\
		u(0)=\phi \in L^2(\R^d).
	\end{dcases}
\end{equation*}
It is well known that the solution to the linear Schr\"dinger equation on $X = \R^d$, denoted $e^{it\De_{\R^d}}\phi$, satisfies the (one-body) Strichartz estimate
\begin{equation}\label{eq:Stri}
	\|e^{it\De_{\R^d}}\phi\|_{L^p_t(\R;L^q_x(\R^d))} \ls \|\phi\|_{L^2(\R^d)}
\end{equation}
for appropriate pairs of the exponents $(p,q)$. We direct the reader, for example, to \cite{Tao06} for further details.
In \cite{FLLS24}, Frank, Lewin, Lieb, and Seiringer considered the generalisation
\begin{equation}\label{eq:OStri}
	\bigg\|\sum_{n=1}^\I a_n |e^{it\De_{\R^d}}\phi_n|^2\bigg\|_{L^{p/2}_t(\R;L^{q/2}_x(\R^d))} \ls \|a_n\|_{\ell^\al_n}
\end{equation}
where $(\phi_n)_{n=1}^\I \subset L^2(\R^d)$ is an orthonormal system and $(a_n)_{n=1}^\I \in \ell_n^\al$. For $\alpha = 1$, this is a simple consequence of \eqref{eq:Stri} and the triangle inequality. However the assumption that $(\phi_n)_{n=1}^\I\subset L^2(\R^d)$ is an orthonormal system allows this to be extended to appropriate choices of $\al>1$. Since \cite{FLLS24}, these so-called \textit{orthonormal Strichartz estimates}
have been extensively studied, see \cite{BHLNS19, BLN20, BLN21, BKS24, CHP17, CHP18, FMSW26, FS17, Had25, HH25, HY26, Hos24, Hos25AHP, Hos25, Hos25PAA, Nak20, MS25, SBMM24, SB25, Mak26} and the references therein.
\subsubsection{Orthonormal Strichartz estimates for density function}
One can consider the counterpart of \eqref{eq:OStri} on $\T$. Recall the free Schr\"odinger equation on $\T$ from \eqref{free_Schro_T}. By restricting $[0,T] := [0,2\pi]$, we can view $u(t,x)$ as a function on $\T^2 = \T_t \times\T_x$. Then we have the one-body estimates
\begin{align}
	&\|U(t)\phi\|_{L^4_\tx(\T^2)} \ls \|\phi\|_{L^2_x(\T)},\label{eq:L4} \\
	&\|U(t)P_{\le N}\phi\|_{L^6_\tx(\T^2)} \ls_\ep N^\ep \|\phi\|_{L^2_x(\T)}. \label{eq:L6}
\end{align}
See \cite{Z74} for \eqref{eq:L4} and \cite{Bou93} for \eqref{eq:L6}. To generalise \eqref{eq:L4} and \eqref{eq:L6}, it is convenient for us to introduce the notion of the density function for operators $A:L^2(\T) \to L^2(\T)$.
\begin{definition}[Density function]
\label{density_function_definition}
Let $A:L^2(\T)\to L^2(\T)$ be a trace-class operator. Recall the singular value decomposition of $A$ given by \eqref{eq:spectral decomposition}. We define the density of $A$ by
\begin{equation}\label{eq:definition of density}
	\rh_A(x):= \sum_{n=1}^\I a_n \phi_n(x) \ov{\ps_n(x)}.
\end{equation}
\end{definition}
\begin{remark}
  It follows from Lemma \ref{duality_lemma} that $\rh_A$ is well-defined and $A \mapsto \rho_A$ is linear for $A\in \FS^1$. We also note that the density is real-valued for self-adjoint $A \in \FS^1$.
\end{remark}
\begin{remark}
Formally one would like to define the density function of $A$ by $\rh_A = \rh(A) = A(x,x)$, where $A(x,x')$ is the integral kernel of $A$. However, because the subset $\{(x,x') \in \T^2 \mid x=x'\} \subset \T^2$ is a null set, $A(x,x)$ is not well-defined. We note that the right hand side of \eqref{eq:definition of density} is well-defined in $L^1(\T)$ because $A$ is trace-class, $(a_n)_{n=1}^\I \in \ell_n^1$, and $\|\phi_n\|_{L^2_x(\T)} = \|\ps_n\|_{L^2_x(\T)} = 1$ for all $n \in \BN$.
\end{remark}
\begin{remark}
In this paper, we will state our orthonormal Strichartz estimates in terms of the (one-particle) density operator $\gamma$, which will be convenient in applications to the NLSS. In particular, for self-adjoint positive $\gamma_0 \in \mathfrak{S}^\alpha$, by singular value decomposition, we can write 
\begin{equation*}
	\ga_0 = \sum_{n=1}^\I a_n |\phi_n\rg \lg \phi_n|.
\end{equation*}
We can also write
\begin{equation*}
	U(t)\ga_0 U^*(t) = \sum_{n=1}^\I a_n |U(t) \phi_n\rg \lg U(t) \phi_n|, 
	\quad \rh(U(t) \ga_0 U^*(t)) = \sum_{n=1}^\I a_n |U(t)\phi_n|^2.
\end{equation*}
\end{remark}
We are now able to state our generalisations of \eqref{eq:L4} and \eqref{eq:L6}.
\begin{proposition}[Orthonormal version of \eqref{eq:L4}]
\label{prop:L2 OStri 1D torus paraphrase}
	If $\al=1$, then we have 
	\begin{equation}\label{eq:L2 OStri 1D torus paraphrase}
		\|\rh(U(t)\ga_0 U^*(t))\|_{L^2_{t,x}(\T^2)} \ls \|\ga_0\|_{\FS^\al}
	\end{equation}
	for any $\ga_0 \in \FS^\al$.
	Moreover, this is sharp in the sense that \eqref{eq:L2 OStri 1D torus paraphrase} fails if $\al > 1$.
\end{proposition}

\begin{proposition}[Orthonormal version of \eqref{eq:L6}]
\label{thm:Nakamura paraphrase}
	Let $\si\in(0,1/3]$. If $1/\al > 1-\si$, we have
	\begin{equation}\label{eq:OStri 1D torus paraphrase}
		\|\rh(U(t)P_{\le N}\ga_0 P_{\le N} U^*(t))\|_{L^3_{t,x}(\T^2)} \ls_\si N^\si \|\ga_0\|_{\FS^\al}
	\end{equation}
	for any $\ga_0 \in \FS^\al$.
	Moreover, this is almost sharp in the sense that \eqref{eq:OStri 1D torus paraphrase} fails if $1/\al < 1-\si$.	
\end{proposition}
\begin{remark}
Let us stress that Proposition \ref{thm:Nakamura paraphrase} is a special case (corresponding to $d=1$) of a more general result proved by Nakamura in \cite{Nak20} (see also the recent preprint \cite{WWZ25+}). We give only the one dimensional case so that we can contrast our result for the renormalised density we will introduce in Section \ref{SEC:Renorm}.
\end{remark}
\begin{remark}
Let us remark that Propositions \ref{prop:L2 OStri 1D torus paraphrase} and \ref{thm:Nakamura paraphrase} are both reasonably simple to prove. Indeed, we state them primarily so that we can contrast the estimates for the density and the renormalised density.
\end{remark}
\begin{remark}
Since the statements of Propositions \ref{prop:L2 OStri 1D torus paraphrase} and \ref{thm:Nakamura paraphrase} are given in terms of the density function, it is not immediately clear that they are indeed periodic analogues of \eqref{eq:OStri}. However, recalling that
\begin{equation*}
\rh(U(t) \ga_0 U^*(t)) = \sum_{n=1}^\I a_n |U(t)\phi_n|^2,
\end{equation*}
and that 
\begin{equation*}
\|\gamma_0\|_{\FS^\alpha} = \left\|\sum_{n=1}^\infty a_n |\phi_n\rangle \langle \phi_n|\right\|_{\mathfrak{S}^{\alpha}} = \|a_n\|_{\ell_n^\alpha},
\end{equation*}
it follows that they are generalisations of \eqref{eq:OStri}. 
\end{remark}

\subsubsection{Orthonormal Strichartz estimates for the renormalised density}
\label{SEC:Renorm}
In this section, we introduce the renormalisation procedure we will use throughout the paper. We begin with the following definition.
\begin{definition}[Renormalised density function]
Let $A:L^2(\T)\to L^2(\T)$ be a trace-class operator. 
We define the renormalised density function $\vrh_A = \vrh(A)$ by
\begin{equation*}
	\vrh_A(x) := \rh_{A}(x) - \frac{1}{2\pi}\Tr A  \equiv \rh_A(x) - \frac{1}{2\pi}\int_{\T} d x' \rh_A(x').
\end{equation*}
\end{definition}
\begin{remark}
Let us note that we call $\vrh_A$ ``renormalised'' because the quantity $ \Tr A$ is formally infinite if $A$ is not trace-class. We also note that the renormalised density is real-valued for self-adjoint $A \in \FS^1$.
\end{remark}
We can now state our orthonormal Strichartz estimates for the renormalised density. First we have the following orthonormal version of \eqref{eq:L4}.
\begin{theorem}[$L^2$ estimate for the renormalised density]\label{thm:L2 bound renormalised}
	Let $d=1$. If $\al \le 2$, we have
	\begin{equation}\label{eq:L2 bound renormalised}
		\|\vrh(U(t)\ga_0 U^*(t))\|_{L^2_\tx(\T^2)} \ls \|\ga_0\|_{\FS^\al}
	\end{equation}
	for any $\ga_0 \in \FS^\al$.
	This is sharp in the sense that \eqref{eq:L2 bound renormalised} fails if $\al > 2$.
\end{theorem}
\begin{remark}
\label{remark_non_trace}
We note that in Theorem \ref{thm:L2 bound renormalised}, we state \eqref{eq:L2 bound renormalised} holds for all $\ga_0 \in \FS^2$. However, recall that we only defined $\vrh_A$ for $A \in \FS^1$. So $\vrh_A$ with $A \in \FS^2\setminus \FS^1$ is not currently defined. We thus interpret Theorem \ref{thm:L2 bound renormalised} in the following way. Consider the linear map $T:\FS^1 \to L^2_\tx(\T^2)$ defined by $T\ga_0 := \vrh(U(t)\ga_0 U^*(t))$. Since $\FS^1$ is dense in $\FS^2$, and since the inequality \eqref{eq:L2 bound renormalised} holds for all $\ga_0 \in \FS^1$, there is a unique bounded extension $\wt{T}:\FS^2 \to L^2_\tx(\T^2)$ of $T$. We define $\vrh(U(t)\ga_0 U^*(t)):= \wt{T}\ga_0$ for all $\ga_0 \in \FS^2$.
\end{remark}
We also have the following orthonormal analogue of \eqref{eq:L6}.

	\begin{theorem}[$L^3$ estimate for the renormalised density]\label{thm:L3 bound renormalised}
	Let $d=1$.
	If $\alpha \in [3/2,3]$ and $\sigma > 2/3 -1/\alpha$, we have
	    \begin{equation}\label{eq:L3 bound renormalised}
	       \|\vrh(U(t)P_{\le N}\ga_0 P_{\le N} U^*(t))\|_{L^3_\tx(\T^2)}
	        \ls_\si N^\si \|\ga_0\|_{\FS^\al}
        \end{equation}
holds for any $\ga_0 \in \FS^\al$.
	This is (almost) sharp in the sense that \eqref{eq:L3 bound renormalised} fails if $\si < 2/3 - 1/\al$. 
\end{theorem}
\begin{remark}
	Let us briefly compare Theorem \ref{thm:L3 bound renormalised} with Proposition \ref{thm:Nakamura paraphrase}. When $0 < \si < 1/3$,  Theorem \ref{thm:L3 bound renormalised} allows $\al$ such that $1/\al > 2/3 - \si$, which is much better than $1/\al > 1-\si$ in Proposition \ref{thm:Nakamura paraphrase}.
\end{remark}

%

\subsubsection{Higher dimensional case}
In this paper, we mainly analyse the case of the one dimensional torus $X=\T$. However, in Section \ref{sec:higher dimensional case}, we introduce a renormalised density in $\T^d$ for $d\ge 2$ and consider the orthonormal Strichartz estimates for the renormalised density. We give a strong necessity condition for the estimate to hold. Moreover, we give an alternative proof of \cite[Theorem 1.4]{Nak20}. 

\subsection{Applications to NLS systems}
\label{SEC:NLSS:introduction}
We now return to the problem of systems of NLS equations on $\T$ presented in Section \ref{SEC:NLSS:1}. Recall from the introduction the cubic NLS system given by
\begin{equation}\label{eq:NLSS} \tag{NLSS}
	\begin{dcases}
		i\pl_t \ga = [-\De \pm \rh_\ga, \ga], \quad \ga:[0,T] \to \CB(L^2(\T)),&\\
		\ga(0)=\ga_0 \in \CB(L^2(\T)).
	\end{dcases}
\end{equation}
We note that in the particular case that $u$ solves the standard cubic NLS given by
\begin{equation*}
i \partial_t u + \De u = \pm |u|^2u,
\end{equation*}
$\ga(t)=|u(t)\rg \lg u(t)|$ solves \eqref{eq:NLSS}. So \eqref{eq:NLSS} can be seen as a natural generalisation of the cubic NLS to many-particle systems. Given this formulation, a natural question is what is the largest choice of $\alpha$ for which \eqref{eq:NLSS} has a solution. In the case that $\alpha = 1$, this is a simple application of local well-posedness of the (one-body) cubic NLS in \cite{Bou93} and the triangle inequality. We direct the reader for example to \cite{CU24} for such a problem. In recent years, the study of such a question for $\alpha > 1$ has attracted a lot of attention, especially in the context of the application of the orthonormal Strichartz estimates.  For example, see \cite{LS14, LS15, CHP17, CHP18, Had25, HH25, HY26, Hos24, Hos25AHP, Hos25, Hos25PAA, BHS25, PS21, NY25, NY25SIAM, CU24, Dong21}.

Our main application concerns the well-posedness of \eqref{eq:NLSS}. First, we have the following result for when we do not renormalise the density.
\begin{theorem}[Optimal well-posedness for the cubic NLS system]\label{thm:cubic optimal well posedness}
	We have that \eqref{eq:NLSS} is globally well-posed in $\FS^1$.
	If $\al>1$, \eqref{eq:NLSS} is ill-posed in $\FS^\al$ in the sense that the solution map is discontinuous.  
\end{theorem}
{\begin{remark}
Let us briefly remark that the global well-posedness of \eqref{NLSS_introduction} with a focusing potential was recently studied in \cite{Ath26}, which appeared shortly after the first version of this work. In \cite{Ath26}, the author proves that \eqref{NLSS_introduction} is globally well-posed in the space $H^1 \FS^1(\T)$, which is the space of operators whose weak derivate are trace-class. The proof proceeds via Fourier-Galerkin methods, and is rather different to the approach in this paper.
\end{remark}}
At least formally, \eqref{eq:NLSS} is equivalent to the following renormalised cubic NLS system
\begin{equation}\label{eq:NLSS'} \tag{RNLSS}
	\begin{dcases}
		i\pl_t \ga = [- \De \pm \vrh_\ga, \ga], \quad \ga:[0,T] \to \CB(L^2(\T)),&\\
		\ga(0)=\ga_0 \in \CB(L^2(\T))
	\end{dcases}
\end{equation}
because $[(2\pi)^{-1}\Tr \ga, \ga] = 0$ if $\ga \in \FS^1$.
However, from the viewpoint of the well-posedness problem, there is a big difference between \eqref{eq:NLSS} and \eqref{eq:NLSS'}. Indeed, we have the following well-posedness theory for the renormalised system.
\begin{theorem}[Optimal well-posedness for the renormalised cubic NLS system]\label{thm:cubic optimal well posedness renormalised}
	If $\alpha \in [1,2]$ then \eqref{eq:NLSS'} is globally well-posed in $\FS^\al$.
    If $\al>2$ then \eqref{eq:NLSS'} is ill-posed in $\FS^\al$ in the sense that the solution map is discontinuous.  
\end{theorem}
\begin{remark}
For the precise statements of Theorems \ref{thm:cubic optimal well posedness} and \ref{thm:cubic optimal well posedness renormalised}, we direct the reader to Propositions \ref{prop:cubic NLSS wellposed} and \ref{prop:cubic NLSS renormalised wellposed}.  We also direct the reader to Definitions \ref{def:solution} and \ref{def:solution renormalised} for precise definitions of solutions to the NLS system.
\end{remark}

\begin{remark}
A related question is the problem of the well-posedness of \eqref{eq:NLSS} and \eqref{eq:NLSS'} when the nonlinearity is nonlocal. In other words, when the right hand side of the PDE is of the form $[-\De + w*\rho_\gamma,\gamma]$, where $w$ is a real-valued even interaction potential. We note that formally, \eqref{eq:NLSS} corresponds to taking $w =  \pm \delta$.
We expect that in this setting, one can obtain the well-posedness portion of Theorems \ref{thm:cubic optimal well posedness} and \ref{thm:cubic optimal well posedness renormalised} whenever $\widehat{w} \in \ell^\infty$. The ill-posedness is less clear, and it would be interesting to see if the nonlocal interaction can improve the well-posedness theory of \eqref{eq:NLSS} and \eqref{eq:NLSS'}. This is beyond the scope of this paper, and we leave it to future research.
\end{remark}
\begin{remark}
Another possible application for Theorem \ref{thm:L3 bound renormalised} is to study the quintic NLS system, given by
\begin{equation}
	\begin{dcases}
		i\pl_t u_n + \De u_n = \pm \rho^2 u_n, \quad u_n :[0,T]\times\T \to \C & \\
        \rh = \sum_{m\in\BN} \lm_m |u_m|^2, \\
		u_n(0)=\phi_n \in H^s(\T)&.
	\end{dcases}
\end{equation}
Note that in the quintic case, we should take our initial data to be in $H^s(\T)$ for $s >0$. This is in light of the work of Kishimoto \cite{Kis14}, which shows that the solution map of the (one-body) quintic NLS on $\T$ cannot be smooth when acting on $L^2(\T)$. The (local) solution map is smooth on $H^s(\T)$ for $s > 0$, see \cite{Bou93}. As in the cubic case, we can restate the problem in terms of density operators by writing
\begin{equation}\label{eq:qNLSS} \tag{qNLSS}
	\begin{dcases}
		i\pl_t \ga = [-\De_x \pm (\rh_\ga)^2, \ga], \quad \ga:[0,T] \to \CB(H^s(\T)),&\\
		\ga(0)=\ga_0 \in \CB(H^s(\T)).
	\end{dcases}
\end{equation}
The difference in this setting compared to the cubic system is that instead of the equation with the renormalised density giving the same NLS system for trace-class operators, one instead obtains a cubic-quintic NLS system. Following the principles of local well-posedness for the one-body equation, as for example in \cite{Zha06,LYYZ25,ABR26}, one expects the system's local well-posedness theory to be determined by the higher-order nonlinearity. Thus, a possible approach to this problem is to subtract the extra cubic term which arises from the commutator, and to apply a perturbation-like argument. We do not comment further on the quintic problem in this paper.
\end{remark}
\subsection{Method of proof}
Let us briefly comment on the method of proof. The main advantage of considering the renormalised density is that it allows us to reduce the set of functions that we test against when using duality arguments. Indeed, we note that for a function $f \in L^2(\T)$ satisfying
\begin{equation*}
\wh{f}(0) =  \frac{1}{\sqrt{2\pi}} \int dx \, f(x) = 0,
\end{equation*}
one has by Parseval's theorem that
\begin{equation*}
\|f\|_{L^2} = \sup_{\substack{g \in L^2 \\ \|g\|_{L^2} = 1 \\\langle g \rangle_x = 0}} \int_{\T} dx \, \overline{g(x)}f(x).
\end{equation*}
This rather simple observation is crucial to a number of the arguments when proving orthonormal Strichartz estimates, and is the primary mechanism which allows us to improve the estimates compared to when one does not implement renormalisation. {It is central to the proof of Lemma \ref{lem:L2 dual estimate}, a key step in the proof of Theorem \ref{thm:L2 bound renormalised}. The proof of the $L^3$ orthonormal Strichartz estimate in Theorem \ref{thm:L3 bound renormalised} is based on a counting estimate and interpolation. }
 
For the proof of the well-posedness of cubic NLS systems, we reduce the argument to estimates of the (renormalised) density. The original density operator $\ga(t)$ can be recovered from the (renormalised) density.  This idea has been used in the context of the Hartree equation for infinitely many particles. See, for example, \cite{LS14,CHP18,Had25,HH25,BHS25}.
This requires us to prove Strichartz estimates for a linear Schr\"odinger equation with potential. Ill-posedness is shown by finding a sequence of initial conditions which contradict the existence of a continuous flow map for a positive time $T > 0$.
\subsection{Outline of the paper}
In Section \ref{SEC:Prelims}, we recall some useful results from operator theory that will be used throughout the paper. Section \ref{sec:OStri} is dedicated to the proof of the orthonormal Strichartz estimates. Section \ref{SEC:UV:STRICHARTZ} develops Strichartz estimates for the propagator $U_V$ and in Section \ref{SEC:WP:proof} we address the problem of the solution theory for the two different NLS systems. We briefly comment on the case of the renormalised density on higher dimensional toruses in Section \ref{sec:higher dimensional case}. Finally, in Appendix \ref{SEC:uv:construction}, we give a construction of the propagator $U_V$ for $V \in L^2_{t,x}$. 
\section{Preliminary results}
\label{SEC:Prelims}
In this section we recall some results from operator theory that we will use throughout the paper. Firstly, we have the following version of H\"older's inequality for Schatten spaces, see \cite{Sim05}.
\begin{lemma}
\label{Holder_inequality}
Let $p,p_1,p_2 \in [1,\infty]$ be such that $\frac{1}{p} = \frac{1}{p_1} + \frac{1}{p_2}$. Suppose that $A_1 \in \FS^{p_1}$ and $A_2 \in \FS^{p_2}$. Then
\begin{equation*}
\|A_1A_2\|_{\FS^p} \leq \|A_1\|_{\FS^{p_1}}\|A_2\|_{\FS^{p_2}}.
\end{equation*}
\end{lemma}
\begin{remark}
Where the context is clear, we will slightly abuse convention and describe both the standard version of H\"older's inequality and Lemma \ref{Holder_inequality} by the name H\"older's inequality.
\end{remark}
We recall the following property of operators that are Hilbert-Schmidt.
\begin{lemma}\label{lem:HS}
Let $\Omega$ be a $\sigma$-finite measure space.  
    For any $A \in \FS^2(L^2(\Om))$, we have
    \begin{equation}
        \|A\|_{\FS^2(L^2(\Om))} = \|A(x,x')\|_{L^2(\Om\times\Om)},
    \end{equation}
    where $A(x,x')$ is the integral kernel of $A$.
\end{lemma}
We have the following duality result. For a proof, see \cite[Section 2.3]{FLLS24}.
\begin{lemma}
\label{duality_lemma}
Suppose $A \in \FS^1$ and $V \in L^\I_x(\T)$. Then
\begin{equation*}
\label{eq:dual}
	\int_\T dx \, \rh_A(x)V(x) = \Tr(A V) \equiv \Tr_{L^2} (A M_V),
\end{equation*}
where $M_V$ denotes the operator which acts as multiplication by $V$.
\end{lemma}
Finally, we have the following key duality result, which is central to our analysis of the renormalised density.
\begin{lemma}
\label{duality_zero_mean_lemma}
Let $g \in L^2_{t,x}([0,T] \times\T)$ be such that $\langle g \rangle = 0$. Then
\begin{equation*}
\|g\|_{L^2_\tx([0,T] \times \T)}
		 =  \sup_{\substack{\|V\|_{L^2_\tx([0,T] \times \T)} \le 1 \\ \lg V \rg = 0}} \left|{\int_{[0,T]}\int_{\T} dx \, dt \, g(t,x)V(t,x)}\right|.
\end{equation*}
\end{lemma}
\begin{proof}
The result is a consequence of the fact that $L^2([0,T] \times \T) = L^2_0([0,T] \times \T) \oplus^\perp \C$. Here $L^2_0([0,T] \times \T)$ is the space of mean-zero functions in $L^2([0,T] \times \T)$.
\end{proof}
\section{Orthonormal Strichartz estimates}\label{sec:OStri}
In this section, we give proofs of the orthonormal Strichartz estimates stated in Propositions \ref{prop:L2 OStri 1D torus paraphrase} and \ref{thm:Nakamura paraphrase} and Theorems \ref{thm:L2 bound renormalised} and \ref{thm:L3 bound renormalised}.
\subsection{Proofs of Propositions \ref{prop:L2 OStri 1D torus paraphrase} and \ref{thm:Nakamura paraphrase}}
We first prove Proposition \ref{prop:L2 OStri 1D torus paraphrase}.
\begin{proof}[Proof of Proposition \ref{prop:L2 OStri 1D torus paraphrase}]
The case for $\alpha = 1$ follows by applying the triangle inequality and applying \eqref{eq:L4}. To show that the estimate fails for $\alpha > 1$, we recall the basis functions $e_n$ defined in \eqref{en_defintion}. It follows that $U(t) e_n = e^{-in^2 t} e_n$. So if $a_n > 0$, one has
\begin{equation*}
    		\bigg\|\sum_{n=1}^\I a_n |U(t) e_n|^2\bigg\|_{L^2_\tx(\T^2)} = \|a_n\|_{\ell_n^1}.
\end{equation*}
It follows that we cannot improve \eqref{eq:L2 OStri 1D torus paraphrase} to any $\alpha > 1$.
\end{proof}
We now give a proof of the sufficiency portion of Proposition \ref{thm:Nakamura paraphrase}. As mentioned previously, this is a special case of a result previously proved by Nakamura in \cite{Nak20}. The necessity portion of the result can be proved as in the original argument in \cite{Nak20}. First we prove the following lemma.
\begin{lemma}[Pointwise estimate]\label{lem:pointwise}
	We have
	\begin{equation}\label{eq:pointwise}
		\|\rh(P_{\le N}AP_{\le N})\|_{L^\I_x(\T)} \ls N\|A\|_{\CB(L^2(\T))}.
	\end{equation}
\end{lemma}
\begin{proof}
	Suppose without loss of generality that $V \in C^\I(\T)$ is such that $V(x)\ge 0$ for all $x \in \T$. Using Lemma \ref{duality_lemma}, cyclicity of the trace, and H\"older's inequality, we have
	\begin{align}
		\abs{\int_{\T} dx \, V(x) \rh(P_{\le N} A P_{\le N})(x)}
		= \abs{\Tr (P_{\le N} V P_{\le N} A)} \le \|P_{\le N} V P_{\le N}\|_{\FS^1} \|A\|_{\CB(L^2(\T))}.
	\end{align}
	The integral kernel of $P_{\le N}$ is given by $P_{\le N}(x,x')=(2\pi)^{-1} \sum_{n=-N}^N e^{in(x-x')}$.
	Therefore
	\begin{align}
		\|P_{\le N} V(x) P_{\le N}\|_{\FS^1}
		= \Tr(P_{\le N} V P_{\le N}) = \Tr(P_{\le N} |V|) 
		&= \frac{1}{2\pi}\sum_{n=-N}^N \int_{\T} dx \, |V(x)| e^{in(x-x)} \\&
        \sim N \|V\|_{L^1_x(\T)},
	\end{align}
    where the second equality is a consequence of the positivity of $V$. It follows that
	\begin{equation}
		\left|\int_{\T} dx \, V(x) \rho(P_{\le N} A P_{\le N})(x) \right|\ls N \|V\|_{L^1_x(\T)} \|A\|_{\CB(L^2(\T))}.
	\end{equation}
	By a duality argument, we obtain \eqref{eq:pointwise}.
\end{proof}
\begin{proof}[Proof of Proposition \ref{thm:Nakamura paraphrase}]
Recall we only prove the sufficiency portion of Proposition \ref{thm:Nakamura paraphrase} here. Applying Proposition \ref{prop:L2 OStri 1D torus paraphrase}, we have
\begin{equation}\label{eq:triangle}
	\|\rh(U(t)P_{\le N}\ga_0 P_{\le N}U^*(t))\|_{L^2_\tx(\T^2)}\ls \|P_{\le N}\ga_0P_{\le N}\|_{\FS^1} \le \|\ga_0\|_{\FS^1}
\end{equation}
Moreover, it follows from Lemma \ref{lem:pointwise} that
\begin{equation}\label{eq:pointwise 2}
	\|\rh(U(t)P_{\le N}\ga_0 P_{\le N}U^*(t))\|_{L^\I_\tx(\T^2)} \ls N\|\ga_0\|_{\CB(L^2(\T))}.
\end{equation}
Interpolating between \eqref{eq:triangle} and \eqref{eq:pointwise 2}, we obtain
\begin{equation}\label{eq:L3 1/3 3/2}
	\|\rh(U(t)P_{\le N}\ga_0 P_{\le N}U^*(t))\|_{L^3_\tx(\T^2)} \ls N^{1/3}\|\ga_0\|_{\FS^{3/2}}.
\end{equation}
Applying the triangle inequality and \eqref{eq:L6}, we have
\begin{equation}\label{eq:L3 epsilon 1}
	\|\rh(U(t)P_{\le N}\ga_0 P_{\le N}U^*(t))\|_{L^3_\tx(\T^2)} \ls_\ep N^\ep\|\ga_0\|_{\FS^1}.
\end{equation}
Interpolating between \eqref{eq:L3 1/3 3/2} and \eqref{eq:L3 epsilon 1} completes the proof of Proposition \ref{thm:Nakamura paraphrase}.
\end{proof}

\subsection{Proof of Theorem \ref{thm:L2 bound renormalised}}
In this section, we prove Theorem \ref{thm:L2 bound renormalised}.
We need the following lemma.
\begin{lemma}\label{lem:L2 dual estimate}
	For any $V\in L^2_\tx(\T^2)$ such that $\lg V \rg = 0$, we have
	\begin{equation}\label{eq:L2 dual estimate}
		\No{\int_\T dt \, U^*(t)V(t)U(t)}_{\FS^2} \ls \|V\|_{L^2_\tx(\T^2)}.
	\end{equation}
\end{lemma}

\begin{proof}[Proof of Lemma \ref{lem:L2 dual estimate}]
	For an operator $A:\ell_m^2 \to \ell_n^2$ with integral kernel $A(n,m)$, one has $\|A\|_{\FS^2(\ell^2_n)} = \|A(n,m)\|_{\ell^2_{n,m}}$. Therefore 
	\begin{align}
		\No{\int_\T dt \, U^*(t) V(t) U(t)}_{\FS^2}
		&=\No{\int_\T dt \, e^{-it(m^2 - n^2)} \wh{V}(t,n-m) }_{\ell_{n,m}^2} \\
		&\sim \|\wt{V}(m^2-n^2, n-m)\|_{\ell^2_{n,m}},
	\end{align}
	where $\wt{V}(n,m):= (2\pi)^{-1}\iint_{\T^2}  dx \, dt \, e^{-int-imx} V(t,x)$ is the space-time Fourier coefficient of $V$. Then
\begin{align}
		&\|\wt{V}(m^2-n^2, n-m)\|_{\ell^2_{n,m}}^2 = \|\wt{V}(-(n+2m)n, n)\|_{\ell^2_m \ell_n^2}^2
		 = \sum_{n,m\in\Z} |\wt{V}(-(n+2m)n, n)|^2 \\
		&\quad \le \sum_{n,m\in\Z}|\wt{V}(mn, n)|^2
		\le  \sum_{n,m\in\Z}|\wt{V}(m, n)|^2 = \|\wt{V}(m, n)\|_{\ell^2_{n,m}}^2 = \|V\|_{L^2_\tx(\T^2)}^2,
	\end{align}
	where we use $\wt{V}(0,0) = \langle V \rangle_{t,x} =0$ to get the second inequality.
	Indeed if $|\wt{V}(0,0)|>0$, then
	$$\sum_{n\in\Z} \sum_{m\in\Z}|\wt{V}(mn, n)|^2 \ge \sum_{m\in\Z}|\wt{V}(0,0)|^2 = \I.$$
\end{proof}
We can now prove Theorem \ref{thm:L2 bound renormalised}.
\begin{proof}[Proof of Theorem \ref{thm:L2 bound renormalised}]
We first prove the sufficiency portion of the result. Fix the notation
\begin{equation*}
\ga_0 := \sum_{n=1}^\I a_n |\phi_n \rg \lg\phi_n|, \quad \varrho := \vrh(U(t)\ga_0 U^*(t)) = \sum_{n=1}^\I a_n |U(t)\phi_n|^2 - \frac{1}{2\pi} \sum_{n=1}^\I a_n.
\end{equation*}
		Since $\lg\varrho\rg = 0$ we have
		\begin{align}
			\|\varrho\|_{L^2_\tx(\T^2)}
			&= \sup \abs{\iint_{\T^2}  dx \, dt \,  V(t,x)\varrho(t,x)}
			= \sup \abs{\iint_{\T^2} dx \, dt \, V(t,x) \rh(U(t)\ga_0 U^*(t))}\\
			&= \sup \abs{\Tr\Dk{\int_{\T} dt \, U^*(t) V(t) U(t) \ga_0}}
			\le \|\ga_0\|_{\FS^2} \sup \No{\int_\T dt \, U^*(t) V(t) U(t)}_{\FS^2},
		\end{align}
		where we have used Lemma \ref{duality_zero_mean_lemma}, Lemma \ref{eq:dual}, and H\"older's inequality. We have also used the cyclicity of the trace. Here the suprema are over the set
        \begin{equation*}
            \{V \in L^{2}_{t,x} :\|V\|_{L^2_\tx(\T^2)} \le 1 \text{ and } \lg V \rg = 0\}.
        \end{equation*}
		Applying Lemma \ref{lem:L2 dual estimate} we obtain
		\begin{equation}
			\|\varrho\|_{L^2_\tx(\T^2)} \ls \|\ga_0\|_{\FS^2}.
		\end{equation}

To prove necessity, we again consider the basis functions $e_n$ defined in \eqref{en_defintion}. Set $\phi_n(x)=2^{-1/2} (e_{n}(x) + e_{-n}(x))$. We note that $\phi_n$ forms an orthonormal system. One can compute that
\begin{equation*}
U(t)\phi_n(x) = \frac{e^{-in^2t}}{\sqrt{2}}(e_{n}(x) + e_{-n}(x)).
\end{equation*}
It follows that
\begin{equation}
	|U(t)\phi_n(x)|^2 = \frac{1}{2\pi}(1+\cos(2nx)).
\end{equation}
Let us fix $N \in \BN$. Then consider the sequence that is $a_n := 1$ if $1\le n \le N$ and is $a_n := 0$ if $n > N$. We compute that $\|a_n\|_{\ell_n^\al}=N^{1/\al}$. We also have that
	\begin{equation}
		\varrho = \sum_{n=1}^N |U(t) \phi_n |^2 - \frac{N}{2\pi} = \frac{1}{2\pi}\sum_{n=1}^N \cos(2nx).
	\end{equation}
Since $(\cos(2nx))_{n=1}^N$ is an orthogonal system in $L^2_\tx(\T^2)$, we have
\begin{equation}
	\|\varrho\|_{L^2_\tx(\T^2)} = \frac{1}{2\pi}\K{\sum_{n=1}^N \|\cos(2nx)\|_{L^2_\tx(\T^2)}^2 }^{1/2}\sim N^{1/2}.
\end{equation} 
So if \eqref{eq:L2 bound renormalised} were to hold for $\al>2$, we would have
\begin{equation}
	N^{1/2}\sim \|\varrho\|_{L^2_\tx(\T^2)} \ls N^{1/\al},
\end{equation}
which cannot hold as $N \to \I$.
\end{proof}

\subsection{Proof of Theorem \ref{thm:L3 bound renormalised}}
In this section we will prove Theorem \ref{thm:L3 bound renormalised}.
We first prove the following key lemma.
\begin{lemma}\label{lem:L2x to L2tx}
	Let $\ep>0$. Then, for any $\ga_0 \in\FS^2$ and $f\in L^2_x(\T)$, we have
	\begin{equation}
		\|\vrh(U(t)\PN \ga_0 \PN U^*(t)) U(t)\PN f\|_{L^2_\tx} \ls_\ep N^\ep \|\ga_0\|_{\FS^2} \|f\|_{L^2_x(\T)}.
	\end{equation}
\end{lemma}

\begin{proof}	
	Let $\varrho(t,x) = \vrh(U(t)\PN \ga_0 \PN U^*(t))$ and $f \in L^2_x(\T)$.
	Then we can write
	$$
	\varrho(t,x) = \frac{1}{2\pi}\sum_{\substack{|m|,|n|\leq N\\m\neq n}} c_{m,n} e^{i(m-n)x-it(m^2-n^2)}, 
	\quad 
	U(t) \PN f(x) =\frac{1}{\sqrt{2\pi}} \sum_{|l|\le N} a_l e^{ilx-itl^2},
	$$
	where $c_{m,n}$ are the Fourier coefficients of $\gamma_0$. Therefore
	\begin{equation}
		\Ps(t,x):= \varrho(t,x)U(t)\PN f(x)
		= \frac{1}{(2\pi)^{3/2}}
		\sum_{\substack{m\neq n\\ |m|,|n|,|l|\leq N}}
		c_{m,n}a_l
		e^{i(m-n+l)x-it(m^2-n^2+l^2)}.
	\end{equation}
	and
	\begin{equation}\label{eq:3.5}
		\|\Ps\|_{L^2_\tx(\T^2)}^2 = \sum_{(\om,k)\in\Z^2} |\wh{\Ps}(\om,k)|^2, \quad \wh{\Ps}(\om,k)=\sum_{\substack{m\neq n\\
				\om=m^2-n^2+l^2\\
				k=m-n+l}}
		c_{m,n}a_l.
	\end{equation}
	We estimate these coefficients by separating into two cases.
	Set $q:=\omega-k^2$.
	If a triple $(m,n,l)$ contributes to \eqref{eq:3.5}, then
	$$\omega=m^2-n^2+l^2, \quad
	k=m-n+l,$$
	Therefore
	\begin{align}\label{eq:3.6}
		q=m^2-n^2+l^2-(m-n+l)^2=2(m-n)(n-l).
	\end{align}
	
	\noindent \textbf{Resonant case $q=0$.}
	Since $m\neq n$, \eqref{eq:3.6}
	implies $n=l$.
	Then  $\om=m^2-n^2+l^2=m^2$ and $k=m-n+l=m$.
	Applying Cauchy--Schwarz, the contribution of the resonant part is
	\begin{equation}
		\sum_{\substack{(\om,k)\in\Z^2 \\ \omega - k^2 = 0}}|\wh{\Ps}(\om,k)|^2
		= \sum_{m\in \Z} \abs{\sum_n c_{m,n}a_n}^2
		\le \|c\|_{\ell_{n,m}^2}^2 \|a\|_{\ell_n^2}^2 = \|\ga_0\|_{\FS^2}^2 \|f\|_{L^2_x}^2.
	\end{equation}
	
	\noindent \textbf{Non-resonant case $q\ne0$.}
	Let $p:=m-n \ne 0$. By \eqref{eq:3.6}, $q=2p(n-l)$, and one has that $p\mid q/2$. Moreover, once $p$ is fixed, the whole triple $(m,n,l)$ is determined by $(\om,k)\in\Z^2$.
	Indeed, since $l=k-p$ and $n-l=q/(2p)$, we have
	$$n=k-p+\frac{q}{2p},\quad m=n+p.$$
	Therefore for fixed $(\omega, k)$, the number of triples $(m,n,l)$ contributing to \eqref{eq:3.5} is bounded by the number of
	divisors of $q/2$.
	Since $|m|,|n|,|l|\le N$, we know that
	$$|q|=|2(m-n)(n-l)| \ls N^2$$
	and that
	\begin{equation}\tag{eq:3.8}
		\#\Ck{(m,n,l): (m,n,l)\text{ contributes to }\wh{\Ps}(\om,k)}
		\ls d(N^2) \ls_\ep N^\ep,
	\end{equation}
	where $d$ is the divisor function.
	Using Cauchy--Schwarz in each $(\om,k)$ and then summing over all $(\om,k)$, we obtain
	\begin{align}\label{eq:3.9}
		\sum_{(\om,k)\in\Z^2, \; q\ne 0}
		|\wh{\Ps}(\om,k)|^2
		&= 
		\sum_{(\om,k)\in\Z^2, \; q\ne 0}
		\bigg|
		\sum_{\substack{m\neq n,\\
				k=m-n+l\\
				\om=m^2-n^2+l^2}}
		c_{m,n}a_l
		\bigg|^2 \\
		&\ls_\ep
		N^\ep
		\sum_{(\om,k)\in\Z^2, \; q\ne 0}
		\sum_{\substack{m\neq n,\\
				k=m-n+l\\
				\om=m^2-n^2+l^2}}
		|c_{m,n}|^2|a_l|^2                                  \\
		&\le
		N^\ep
		\sum_{m,n,\ell}|c_{m,n}|^2|a_l|^2  
		= N^\ep \|\ga_0\|_{\FS^2}^2 \|f\|_{L^2_x}^2.
	\end{align}
\end{proof}
By Lemma \ref{lem:L2x to L2tx}, we obtain the endpoint estimate. 
\begin{lemma}\label{lem:L3 endpoint}
	Let $\ep>0$. Then we have
	\begin{equation}\label{eq:L3 endpoint}
		\|\vrh(U(t)P_{\le N} \ga_0 P_{\le N} U^*(t))\|_{L^3_\tx(\T^2)} \ls_\ep N^\ep \|\ga_0\|_{\FS^{3/2}}.
	\end{equation}
\end{lemma}
\begin{proof}
	Let
	\begin{equation}
		T_N(\ga_0) := \vrh(U(t)P_{\le N} \ga_0 P_{\le N} U^*(t)).
	\end{equation}
	It suffices to prove the trilinear estimate
	\begin{equation}
		\|T_N(\ga_0) T_N(\ga_1) T_N(\ga_2)\|_{L^1_\tx(\T^2)} \ls_\ep N^\ep \|\ga_0\|_{\FS^{3/2}} \|\ga_1\|_{\FS^{3/2}} \|\ga_2\|_{\FS^{3/2}}.
	\end{equation}
	We can write
	\begin{align}
		\|T_N(\ga_0) T_N(\ga_1) T_N(\ga_2)\|_{L^1_\tx(\T^2)}
		&\le \|\rh(U(t)P_{\le N}\ga_0 P_{\le N} U^*(t)) T_N(\ga_1) T_N(\ga_2)\|_{L^1_\tx(\T^2)} \\
		&\quad + \frac{1}{2\pi} |\Tr(P_{\le N}\ga_0 P_{\le N})| \| T_N(\ga_1) T_N(\ga_2)\|_{L^1_\tx(\T^2)} =: \SA + \SB.
	\end{align}
	For $\SA$, Lemma \ref{lem:L2x to L2tx} implies that
	\begin{align}
		\SA &= \iint_{\T^2} \rh(U(t)P_{\le N} \ga_0 P_{\le N} U^*(t)) T_N(\ga_1)(t,x) e^{i\theta(t,x)} T_N(\ga_2)(t,x) dxdt \\
		&= \Tr\Dk{P_{\le N}\int_\T U^*(t) \K{T_N(\ga_1)(t,x) e^{i \theta(t,x)} T_N(\ga_2)(t,x)} U(t)dt P_{\le N} \ga_0}  \\
		&\le \prod_{j=1,2} \No{T_N(\ga_j)(t,x)U(t)P_{\le N}}_{L^2_x \to L^2_\tx} \|\ga_0\|_{\FS^1}
		\ls_\ep N^\ep   \|\ga_0\|_{\FS^1} \|\ga_1\|_{\FS^2} \|\ga_2\|_{\FS^2}, 
	\end{align}
	where $\th(t,x)$ is a proper real-valued function.
	For $\SB$, it follows from Theorem \ref{thm:L2 bound renormalised} that 
	\begin{equation}
		\SB \ls \|\ga_0\|_{\FS^1} \|T_N(\ga_1)\|_{L^2_\tx(\T^2)} \|T_N(\ga_2)\|_{L^2_\tx(\T^2)} \ls \|\ga_0\|_{\FS^1} \|\ga_1\|_{\FS^2} \|\ga_2\|_{\FS^2}.
	\end{equation}
	Therefore
	\begin{equation}\label{eq:trilinear 1}
		\|T_N(\ga_0)T_N(\ga_1)T_N(\ga_2)\|_{L^1_\tx(\T^2)} \ls_\ep N^\ep \|\ga_0\|_{\FS^1} \|\ga_1\|_{\FS^2} \|\ga_2\|_{\FS^2}.
	\end{equation}
	Arguing similarly, we obtain
	\begin{align}
		&\|T_N(\ga_0)T_N(\ga_1)T_N(\ga_2)\|_{L^1_\tx(\T^2)} \ls_\ep N^\ep \|\ga_0\|_{\FS^2} \|\ga_1\|_{\FS^1} \|\ga_2\|_{\FS^2}, \label{eq:trilinear 2}\\
		&\|T_N(\ga_0)T_N(\ga_1)T_N(\ga_2)\|_{L^1_\tx(\T^2)} \ls_\ep N^\ep \|\ga_0\|_{\FS^2} \|\ga_1\|_{\FS^2} \|\ga_2\|_{\FS^1}. \label{eq:trilinear 3}
	\end{align}
	By interpolating between \eqref{eq:trilinear 1}, \eqref{eq:trilinear 2}, and \eqref{eq:trilinear 3}, we obtain \eqref{eq:L3 endpoint}.
\end{proof}
We now prove Theorem \ref{thm:L3 bound renormalised}.
\begin{proof}[Proof of Theorem \ref{thm:L3 bound renormalised}]
We recall the notation
\begin{equation*}
\ga_0 := \sum_{n=1}^\I a_n |\phi_n \rg \lg\phi_n|, \quad \varrho := \vrh(U(t)P_{\le N}\ga_0 P_{\le N}U^*(t)).
\end{equation*}
We first prove the sufficiency portion of the theorem. We interpolate between \eqref{eq:pointwise} and \eqref{eq:L2 bound renormalised} to get
\begin{equation}\label{eq:L3 1/3 3}
	\|\varrho\|_{L^3_\tx(\T^2)} \ls N^{1/3} \|\ga_0\|_{\FS^3}.
\end{equation}
Interpolating between \eqref{eq:L3 1/3 3} and \eqref{eq:L3 endpoint} yields \eqref{eq:L3 bound renormalised}.

We now prove the necessity portion of the theorem. The proof is similar to the necessity portion of Theorem \ref{thm:L2 bound renormalised}, but for clarity, we include the details. Recall the basis functions $e_n$ defined in \eqref{en_defintion}. Recall the orthonormal system defined by $\phi_n(x)=2^{-1/2} (e_{n}(x) + e_{-n}(x))$. We consider the sequence $a_n := 1$ if $1\le n \le N$ and $a_n := 0$ otherwise. Then for $1\le n \le N$, we have
\begin{equation}
    U(t)P_{\le N}\phi_n(x) = U(t)\phi_n(x)  = \frac{e^{-in^2t}}{\sqrt{2}}(e_{n}(x) + e_{-n}(x)).
\end{equation}
It follows that
\begin{align}
	|U(t)P_{\le N}\phi_n(x)|^2 = \frac{1}{2\pi}(1+\cos(2nx)) \quad \text{ and } \quad \|P_{\le N}\phi_n(x)\|_{L^2_x(\T)}^2 = 1,
\end{align}
which imply
\begin{equation}
\varrho = \frac{1}{2\pi}\sum_{n=1}^N \cos(2nx)= \frac{1}{4\pi}\K{\frac{{\sin( (2N+1)x)}}{\sin(x)} - 1}.
\end{equation}
Therefore, we obtain
\begin{align}
	\|\varrho\|_{L^3_\tx(\T^2)}
	&\gs \No{\frac{\sin((2N+1)x)}{\sin(x)}}_{L^3_\tx(\T^2)} - 1 \gs N \No{ { \frac{\sin((2N+1)x)}{(2N+1)x}} \frac{x}{\sin(x)}}_{L^3_x([0,(10N)^{-1}])}  - 1\\
	&\gs N \|1\|_{L^3_x([0,(10N)^{-1}])}  - 1 \sim N^{2/3} -1.
\end{align}
Therefore if \eqref{eq:L3 bound renormalised} holds, we have
\begin{equation}
	N^{2/3} - 1 \ls N^{\si+1/\al},
\end{equation} 
which implies $2/3 \le \si + 1/\al$.
\end{proof}

\section{Optimal well-posedness of the NLS system}
\label{SEC:WP:NLSS}
In this section, we prove Theorems \ref{thm:cubic optimal well posedness} and \ref{thm:cubic optimal well posedness renormalised}.
More precisely, we prove Propositions \ref{prop:cubic NLSS wellposed} and \ref{prop:cubic NLSS renormalised wellposed}. Note that Propositions \ref{prop:cubic NLSS wellposed} and \ref{prop:cubic NLSS renormalised wellposed} are concerned with local well-posedness. 
However, global well-posedness immediately follows from the blowup alternatives in these propositions and the unitarity of the evolution $\|\ga(t)\|_{\FS^\al} = \|\ga_0\|_{\FS^\al}$ (see Definitions \ref{def:solution} and \ref{def:solution renormalised}).
Let $E_R^\al:=\{\ga_0:\|\ga_0\|_{\FS^\al}\le R\}$.
\begin{proposition}[Precise statement of Theorem \ref{thm:cubic optimal well posedness}]\label{prop:cubic NLSS wellposed}
\leavevmode
\begin{itemize}
	\item[$(i)$] (Well-posedness.) For any $\ga_0 \in \FS^1$, there exists $T=T(\|\ga_0\|_{\FS^1})>0$ and a unique solution $\ga(t)\in C([0,T];\FS^1)$ to \eqref{eq:NLSS} such that $\rh_\ga \in L^2_t([0,T];L^2_x(\T))$.
	Moreover, the data-to-solution map
	\begin{equation}
		\mathcal{S}: E_R^1 \ni \ga_0 \mapsto (\ga,\rh_\ga) \in C([0,T];\FS^1)\times L^2_t([0,T];L^2_x(\T))
	\end{equation}
	is Lipschitz continuous for any $R>0$ and $T=T(R)>0$. Finally, we have the blowup alternative; we have
	\begin{equation}
		T_{\mathrm{max}} < \I \implies \lim_{t \nearrow T_{\mathrm{max}}} \|\ga(t)\|_{\FS^1} = \I,
	\end{equation}
	where $T_{\text{max}}>0$ is the maximal time that the solution exists.

    \item[$(ii)$] (Ill-posedness.) The data-to-density map constructed above
    \begin{equation*}
    \mathcal{S}_2 :E_R^1 \ni \ga_0 \mapsto  \rho_\ga \in L^2_t([0,T]; L^2_x(\T))  
    \end{equation*}
	cannot be extended to a continuous map $E_R^\al \to L^2_t([0,T]; L^2_x(\T))$ for $\al>1$.

	\end{itemize}
\end{proposition}

\begin{proposition}[Precise statement of Theorem \ref{thm:cubic optimal well posedness renormalised}]\label{prop:cubic NLSS renormalised wellposed}
\leavevmode
	\begin{itemize}
		\item[$(i)$] (Well-posedness.) 	For any $\ga_0 \in \FS^2$, there exists $T=T(\|\ga_0\|_{\FS^2})>0$ and a unique solution $\ga(t)\in C([0,T];\FS^2)$ to \eqref{eq:NLSS'} such that $\vrh_\ga \in L^2_t([0,T];L^2_x(\T))$.
		Moreover, the data-to-solution map
		\begin{equation}
			\mathcal{S}: E_R^2 \ni \ga_0 \mapsto (\ga,\vrh_\ga) \in C([0,T];\FS^2)\times L^2_t([0,T];L^2_x(\T))
		\end{equation}
		is Lipschitz continuous for any $R>0$ and $T=T(R)>0$. Finally, we have the blowup alternative; we have
		\begin{equation}
			T_{\mathrm{max}} < \I \implies \lim_{t \nearrow T_{\mathrm{max}}} \|\ga(t)\|_{\FS^2} = \I,
		\end{equation}
		where $T_{\text{max}}>0$ is the maximal time that the solution exists.
    \item[$(ii)$] (Ill-posedness.) The data-to-density map constructed above
    \begin{equation*}
    \mathcal{S}_2 :E_R^2 \ni \ga_0 \mapsto  \overline{\rho}_\ga \in L^2_t([0,T]; L^2_x(\T))  
    \end{equation*}
	cannot be extended to a continuous map $E_R^\al \to L^2_t([0,T]; L^2_x(\T))$ for $\al>2$.	
	\end{itemize}
\end{proposition}
\begin{remark}
For the precise definitions of solutions, we direct the reader to Definitions \ref{def:solution} and \ref{def:solution renormalised}.
\end{remark}
\subsection{Strichartz estimates for \texorpdfstring{$U_V(t,s)$}{UV(t,s)}}
\label{SEC:UV:STRICHARTZ}
Before proceeding with the proofs of Propositions \ref{prop:cubic NLSS wellposed} and \ref{prop:cubic NLSS renormalised wellposed}, we first need to prove Strichartz estimates for the propagator of the linear Schr\"odinger equation with time-dependent potential $V:[0,T]\times\T \to \R$ given by
\begin{equation}
\label{eq:Schro_pot}
	i\pl_t u + \De u - V(t,x)u = 0, \quad u(s)=\phi.
\end{equation}
Let $u(t)=U_V(t,s)\phi$ denote the solution to the linear equation \eqref{eq:Schro_pot}. Moreover, we will write $U_V(t):= U_V(t,0)$. We note that we have the unique existence of $U_V(t,s)$ for any $V \in L^2_\tx([0,T]\times\T)$ and $t,s\in [0,T]$, see Appendix \ref{sec:appendix}. Before proving Strichartz estimates corresponding to $U_V$, we need the following generalisation of \eqref{eq:L4}.
\begin{lemma}\label{lem:Stri inhomogeneous}
For some $\theta > 0$, we have
	\begin{align}
		&\|U(t)\phi\|_{L^4_\tx([0,T]\times\T)} \ls T^{\theta} \|\phi\|_{L^2_x(\T)}, \label{eq:L4 refine}\\
		&\No{\int_0^t d\ta \, U(t-\ta)u(\ta)}_{C_t([0,T];L^2_x(\T)) \cap L^4_\tx([0,T]\times\T)} \ls T^{\theta} \|u\|_{L^{4/3}_\tx([0,T]\times\T)}.\label{eq:Stri inhomogeneous}
	\end{align}
\end{lemma}
\begin{proof}
	The estimates \eqref{eq:L4 refine} and \eqref{eq:Stri inhomogeneous} are well known, but we briefly outline the proof.
	It suffices to prove \eqref{eq:L4 refine}
	because we can prove \eqref{eq:Stri inhomogeneous} using \eqref{eq:L4 refine} and the standard duality argument with the Christ--Kiselev lemma.
	By \cite[Lemmas 3.10, 3.11, and Theorem 3.25]{ET16}, we have
	\begin{equation}
		\|U(t)\phi\|_{L^4_\tx([0,T]\times\T)}
		\ls \No{U(t)\phi}_{X^{0,3/8}_T} \ls T^{\theta} \|U(t)\phi\|_{X^{0,7/16}_T} \ls T^{\theta}\|\phi\|_{L^2_x(\T)}.
	\end{equation}
	See \cite{ET16} for the definition of $X^{s,b}_T$. We note that \cite[Lemmas 3.10 and 3.11]{ET16} are stated for the KdV equation, however they also hold for the Schr\"{o}dinger equation (see \cite[page 91]{ET16}).
\end{proof}

Now we can proceed with the Strichartz estimate of $U_V$. Throughout the remainder of the paper we take $V$ to be real-valued.
\begin{lemma}\label{lem:Stri V}
	There is a monotone increasing function $\Ph:[0,\I) \to [0,\I)$ and $\theta > 0$ such that the following is true. 
	For any $0\le T \le 2\pi$ and $V \in L^2_\tx([0,T]\times\T)$, we have
	\begin{align}
		\|U_V(t)\phi\|_{L^4_\tx([0,T]\times\T)} &\ls \Ph(T^{\theta}\|V\|_{L^2_\tx([0,T]\times\T)}) T^{\theta}\|\phi\|_{L^2_x(\T)}, \label{eq:Stri V}\\
		\label{eq:Stri V dual}
		\No{\int_0^t d\ta \, U_V(t,\ta)u(\ta)}_{C_t([0,T];L^2(\T)) \cap L^4_\tx([0,T]\times \T)} &\ls \Ph(T^{\theta}\|V\|_{L^2_\tx([0,T]\times\T)}) T^{\theta} \|u\|_{L^{4/3}_\tx([0,T]\times\T)}. 	
	\end{align}
\end{lemma}

\begin{remark}
	In this paper, we always denote a monotone increasing function from $[0,\I)$ to $[0,\I)$ by $\Ph$.  Like a constant $C$, we use $\Ph$ to represent a function that may change line-to-line. The reader can also take $\theta = 1/16$ if they prefer a numerical value.
\end{remark}

\begin{proof}
	It suffices to prove \eqref{eq:Stri V} because we can obtain \eqref{eq:Stri V dual} from \eqref{eq:Stri V} by the standard duality argument.
	Writing $u(t):=U_V(t)\phi$, we have
	\begin{equation}
		u(t)=U(t)\phi - i\int_0^t d\ta \,  U(t-\ta)V(\ta)u(\ta).
	\end{equation}
	Therefore Lemma \ref{lem:Stri inhomogeneous} implies
	\begin{align}
		\|u\|_{L^4_\tx([0,T]\times\T)}
		&\le CT^{\theta}\|\phi\|_{L^2_x(\T)} + CT^{\theta}\|V(t)u(t)\|_{L^{4/3}_\tx([0,T]\times\T)} \\
		&\le CT^{\theta}\|\phi\|_{L^2_x(\T)} + CT^{\theta}\|V\|_{L^2_\tx([0,T]\times\T)}\|u\|_{L^4_\tx([0,T]\times\T)}.
	\end{align}
	If $T^{\theta}\|V\|_{L^2_\tx([0,T]\times\T)} \le \de:=1/(2C)$, then we have
	\begin{equation}
		\|u\|_{L^4_\tx([0,T]\times\T)} \le 2CT^{\theta}\|\phi\|_{L^2_x(\T)}.
	\end{equation}
	When $T^{\theta}\|V\|_{L^2_\tx([0,T]\times\T)}$ is not sufficiently small, we  decompose $[0,T]$ as
	\begin{equation}
		0=T_0 < T_1 < \cdots < T_N <T_{N+1}=T
	\end{equation}
	so that $\de/2\le T^{\theta}\|V\|_{L^2([T_k,T_{k+1}])} \le \de$ for all $k=0,\dots,N-1$ and $T^{\theta}\|V\|_{L^2([T_N,T_{N+1}])} \le \de$.
	Writing $V_k(t):=V(t)\II_{[T_k,T_{k+1}]}(t)$, we have $U_V(t,T_k)=U_{V_k}(t,T_k)$ for all $t \in [T_k,T_{k+1}]$. So for all $t\in[T_k,T_{k+1}]$, we have
	\begin{equation}
		U_V(t)=U_V(t,T_k)U_V(T_k)=U_{V_k}(t,T_k)U_V(T_k) = U_{V_k}(t)U_{V_k}^*(T_k)U_V(T_k).
	\end{equation}
	Therefore
	\begin{align}
		\|U_V(t)\phi\|_{L^4_\tx([0,T]\times\T)} &\le \sum_{k=0}^N \|U_V(t)\phi\|_{L^4_\tx([T_k,T_{k+1}]\times \T)} \\
		&=\sum_{k=0}^N \|U_{V_k}(t)U_{V_k}^*(T_k)U_V(T_k)\phi\|_{L^4_\tx([T_k,T_{k+1}]\times \T)} \\
		&\le 2CN T^{\theta} \|\phi\|_{L^2_x(\T)}.
	\end{align}
	The final inequality uses that $T^{\theta}\|V\|_{L^2([T_k,T_{k+1}])} \le \de$. Finally, since $N \ls \de^{-2}(1+T^{2\theta} \|V\|_{L^2_{t,x}([0,T]\times\T)}^2 )$, we have
	\begin{align}
		\|U_V(t)\phi\|_{L^4_\tx([0,T]\times\T)} &\ls (1+T^{2\theta}\|V\|_{L^2_\tx([0,T]\times\T)}^2) T^{\theta} \|\phi\|_{L^2_x(\T)}.
	\end{align}
\end{proof}
Using Lemma \ref{lem:Stri V} and the triangle inequality, we get the following corollary.
\begin{corollary}\label{cor:Stri V}
	There is a monotone increasing function $\Ph:[0,\I) \to [0,\I)$ such that the following holds.
	For any $0\le T \le 2\pi$ and $V\in L^2_\tx([0,T]\times\T)$, we have
	\begin{align}
		&\|\rh(U_V(t)\ga_0 U_V^*(t))\|_{L^2_\tx([0,T]\times\T)} \le \Ph(T^{\theta}\|V\|_{L^2_\tx([0,T]\times\T)}) T^{2\theta}\|\ga_0\|_{\FS^1}.
	\end{align}
\end{corollary}
We now prove an analogous result for the renormalised density. We obtain the following improved version of Corollary \ref{cor:Stri V}.
\begin{lemma}\label{lem:averaged Strichartz V}
	There is a monotone increasing function $\Ph:[0,\I) \to [0,\I)$ such that for any $0 \le T \le 2\pi$, $\ga_0 \in \FS^2$, and $V \in L^2_\tx([0,T]\times\T)$, 
	we have
	\begin{equation}
		\No{\vrh(U_V(t)\ga_0 U_V^*(t))}_{L^2_\tx([0,T]\times\T)} \ls \Ph(T^{\theta}\|V\|_{L^2_\tx([0,T]\times\T)}) \|\ga_0\|_{\FS^2}.
	\end{equation}
\end{lemma}

\begin{proof}
Since $U_V(t)$ satisfies
	\begin{equation}
		U_V(t) = U(t) - i\int_0^t d\ta \, U(t-\ta)V(\ta)U_V(\ta) =: U(t) + D_V(t),
	\end{equation}
	we have
	\begin{multline}
    \label{decomposition1}
		\vrh(U_V(t)\ga_0 U_V^*(t))= \vrh(U(t)\ga_0 U^*(t)) + \rho(D_V(t)\ga_0 U^*(t)) \\
		\quad  + \rho(U(t)\ga_0 D_V^*(t)) + \rho(D_V(t)\ga_0 D_V^*(t))
		=: \varrho_0 + \varrho_1 + \varrho^1 + \varrho_2.
	\end{multline}
	Here we note that \eqref{decomposition1} contains only one renormalisation since both the left and right hand side of \eqref{decomposition1} require us to subtract $c\mathrm{Tr}(\gamma_0)$ once. For $\varrho_0$, Theorem \ref{thm:L2 bound renormalised} implies
	\begin{equation}
		\|\vrh(U(t)\ga_0 U^*(t))\|_{L^2_\tx([0,T]\times\T)} \ls \|\ga_0\|_{\FS^2}.
	\end{equation}
	For the other terms, we can rearrange \eqref{decomposition1} to write
	\begin{equation}
		\varrho_1 + \varrho^1 + \varrho_2 = \vrh(U_V(t)\ga_0 U_V^*(t)) - \vrh(U(t)\ga_0 U^*(t)),
	\end{equation}
	which implies $\lg \varrho_1 + \varrho^1 + \varrho_2 \rg = 0$. Using Lemma \ref{duality_zero_mean_lemma}, we have
\begin{multline*}
	  \|\varrho_1 + \varrho^1 + \varrho_2\|_{L^2_\tx([0,T]\times\T)} \\
      = \sup \Ck{ \int_0^T\int_\T dx \, dt \, f(t,x) (\varrho_1 + \varrho^1 + \varrho_2)(t,x) }  \le \sup \Ck{ \int_0^T\int_\T dx \, dt \, f(t,x) \varrho_1(t,x)} \\     
 \quad + \sup \Ck{ \int_0^T\int_\T dx \, dt \,  f(t,x) \varrho^1(t,x)}
 + \sup \Ck{ \int_0^T\int_\T dx \, dt \, f(t,x) \varrho_2(t,x) }
        =:\SA_1 + \SA^1 + \SB,
\end{multline*}
	where the suprema are taken over
    \begin{equation*}
        \{f \in L^2_\tx([0,T]\times\T): \|f\|_{L^2_\tx([0,T]\times\T)}\le 1 \text{ and } \lg f \rg = 0\}.
    \end{equation*}
    We first bound $\SA_1$ and $\SA^1$. Since $\SA^1$ can be dealt with in the same way, we only include the details for $\SA_1$. Let $f \in L^2_\tx([0,T]\times\T)$ be such that $\lg f \rg = 0$ and define $V_0 := |V|^{1/2}$ and $V = V_0 V_1$. Applying H\"older's inequality and Lemma \ref{lem:Stri V}, we have
	\begin{multline}
		\abs{\int_0^T \int_\T dx \, dt \,  f(t,x)\varrho_1(t,x)} = \abs{\Tr\Dk{\int_0^T dt \, U^*(t) f(t) U(t) \int_0^t d\ta \, U^*(\ta) V(\ta) U_V(\ta) \ga_0}} \\
		\le \No{\int_0^T dt \, U^*(t) f(t) U(t) \int_0^t d\ta \, U^*(\ta) V(\ta) U_V(\ta)}_{\FS^2} \|\ga_0\|_{\FS^2} \\
		 \le \Ph(T^{\theta}\|V\|_{L^2_\tx([0,T]\times\T)}) T^{\theta}\|V\|_{L^2_\tx([0,T]\times\T)}^{1/2} \\
		\times \No{\int_0^T dt \, U^*(t) f(t) U(t) \int_0^t d\ta \, U^*(\ta) V_0(\ta)}_{\FS^2(L^2_\tx \to L^2_x)} \|\ga_0\|_{\FS^2}.
	\end{multline}
	We note that Lemmas \ref{lem:L2 dual estimate} and \ref{lem:Stri inhomogeneous} imply that
	\begin{equation}
		\No{\int_0^T dt \, U^*(t) f(t) U(t) \int_0^T d\ta \,  U^*(\ta) V_0(\ta)}_{\FS^2(L^2_\tx \to L^2_x)}
		\ls T^{\theta} \|f\|_{L^2_\tx ([0,T]\times\T)} \|V_0\|_{L^4_\tx([0,T]\times\T)}.
	\end{equation}
	Hence it follows from \cite[Corollary 4.9]{BHS25} that
	\begin{equation}\label{eq:partially orthogonal}
		\No{\int_0^T dt \, U^*(t) f(t) U(t) \int_0^t d\ta \, U^*(\ta) V_0(\ta)}_{\FS^2(L^2_\tx \to L^2_x)}
		\ls T^{\theta}\|f\|_{L^2_\tx ([0,T]\times\T)} \|V_0\|_{L^4_\tx([0,T]\times\T)},
	\end{equation}
	where we note that the second integral is now taken over the interval $[0,t]$. Collecting the above estimates and recalling that $\|V_0\|_{L^4_\tx([0,T]\times\T)} = \|V\|_{L^2_\tx([0,T]\times\T)}^{1/2}$, we conclude that
	\begin{equation}
		\|\varrho_1\|_{L^2_\tx([0,T]\times\T)} \le \Ph(T^{\theta} \|V\|_{L^2_\tx([0,T]\times\T)}) \|\ga_0\|_{\FS^2}.
	\end{equation}
	Here we note that we have absorbed the  $T^\th \|V\|_{L^2_\tx([0,T]\times\T)}$ into the function $\Phi$.
    
	We now estimate $\SB$. Let $f \in L^2_\tx([0,T]\times\T)$ such that $\lg f \rg = 0$. Moreover, define $V_0 := |V|^{1/2}$ and $V = V_0 V_1$. By H\"older's inequality and Lemma \ref{lem:Stri V}, we have
	\begin{multline}
		\abs{\int_0^T \int_\T dx \, dt \, f(t,x)\varrho_2(t,x)}\\
		= \abs{\Tr\Dk{\int_0^T d\ta' \, U_V^*(\ta')V(\ta')U(\ta')\int_{\ta'}^T dt \, U^*(t) f(t) U(t) \int_0^t d\ta \, U^*(\ta) V(\ta) U_V(\ta) \ga_0}} \\
		\le \Ph(T^{\theta} \|V\|_{L^2_\tx([0,T]\times\T)}) T^{2\theta} \|V\|_{L^2_\tx([0,T]\times\T)}\|\ga_0\|_{\FS^2}\\
		\times \No{V_1(\ta')U(\ta') \int_{\ta'}^T dt \, U^*(t) f(t) U(t) \int_0^t d\ta \, U^*(\ta)V_0(\ta)}_{\FS^2(L^2_{\ta,x} \to L^2_{\ta',x})}. 
	\end{multline}
	By \eqref{eq:partially orthogonal}, we have
	\begin{multline}
	\label{bound_1}	\No{V_1(\ta')U(\ta') \int_0^T dt \, U^*(t) f(t) U(t) \int_0^t d\ta \, U^*(\ta)V_0(\ta)}_{\FS^2(L^2_{\ta,x} \to L^2_{\ta',x})} \\
		\le \|V_1(t)U(t)\|_{L^2_x \to L^2_\tx} \No{ \int_0^T dt \, U^*(t) f(t) U(t) \int_0^t d\ta \, U^*(\ta)V_0(\ta)}_{\FS^2(L^2_{\ta',x} \to L^2_x)} \\
		\ls T^{2\theta}\|V\|_{L^2_\tx([0,T]\times\T)} \|f\|_{L^2_\tx([0,T]\times\T)}.
	\end{multline}
	By \eqref{eq:partially orthogonal} and the identity $\|A\|_{\FS^\al(H \to K)} = \|A^*\|_{\FS^\al(K \to H)}$, we have
	\begin{multline}
		\No{V_1(\ta')U(\ta') \int_{\ta'}^T dt \, U^*(t) f(t) U(t) \int_0^T d\ta \, U^*(\ta)V_0(\ta)}_{\FS^2(L^2_{\ta,x} \to L^2_{\ta',x})} \\
		 \le \No{V_1(\ta')U(\ta') \int_{\ta'}^T dt \, U^*(t) f(t) U(t)}_{\FS^2(L^2_x \to L^2_{\ta',x})} \No{\int_0^T d\ta \, U^*(\ta)V_0(\ta)}_{\FS^2(L^2_{\ta,x} \to L^2_{x})}  \\
		 \ls T^{\theta}\|V\|_{L^2_\tx([0,T]\times\T)}^{1/2}
		     \No{\int_0^T dt \, U^*(t) f(t)U(t) \int_0^t d\tau' \, U(\ta')V_1(\ta')}_{\FS^2(L^2_{\ta',x} \to L^2_x)} \\
		 \ls T^{2\theta}\|V\|_{L^2_\tx([0,T]\times\T)} \|f\|_{L^2_\tx([0,T]\times\T)},
	\end{multline}
    where we have used Lemma \ref{lem:Stri inhomogeneous}. Writing
	\begin{equation}
		\int_0^T d\ta \int_\ta^{\ta'}dt = \int_0^T dt \int_0^t d\ta - \int_{\ta'}^T dt \int_0^T d\ta,
	\end{equation}
	we have 
	\begin{multline}
	   \No{V_1(\ta')U(\ta')\int_0^T d\ta \K{\int_{\ta}^{\ta'} U^*(t)f(t)U(t)dt} U(\ta)V_0(\ta)}_{\FS^2(L^2_{\ta,x} \to L^2_{\ta',x})} \\
		\ls T^{2\theta}\|V\|_{L^2_\tx([0,T]\times\T)} \|f\|_{L^2_\tx([0,T]\times\T)}.
	\end{multline}
	Using \cite[Lemma 4.1]{BHS25}, we conclude
	\begin{multline}
		\No{V_1(\ta')U(\ta')\int_0^{\ta'}dt \, U^*(t)f(t)U(t)\int_0^t d\ta \, U(\ta)V_0(\ta)}_{\FS^2(L^2_{\ta,x} \to L^2_{\ta',x})} \\
		=\No{V_1(\ta')U(\ta')\int_0^{\ta'}d\ta \K{\int_{\ta}^{\ta'} U^*(t)f(t)U(t)dt} U(\ta)V_0(\ta)}_{\FS^2(L^2_{\ta,x} \to L^2_{\ta',x})} \\
		 \ls T^{2\theta}\|V\|_{L^2_\tx([0,T]\times\T)} \|f\|_{L^2_\tx([0,T]\times\T)}.
	\end{multline}
    Recalling \eqref{bound_1}, it follows that

	\begin{multline}
		\No{V_1(\ta')U(\ta')\int_{\ta'}^T dt \, U^*(t)f(t)U(t)\int_0^t d\ta \, U(\ta)V_0(\ta)}_{\FS^2(L^2_{\ta,x} \to L^2_{\ta',x})} \\
		 \ls T^{2\theta}\|V\|_{L^2_\tx([0,T]\times\T)} \|f\|_{L^2_\tx([0,T]\times\T)}.
	\end{multline}
	Therefore
	\begin{equation}
		\abs{\int_0^T \int_\T dx\, dt \,f(t,x)\varrho_2(t,x)}
		\le \Ph(T^{\theta} \|V\|_{L^2_\tx([0,T]\times\T)}) \|\ga_0\|_{\FS^2},
	\end{equation}
	and duality yields
	\begin{equation}
		\|\varrho_2\|_{L^2_\tx([0,T]\times\T)} \le \Ph(T^{\theta} \|V\|_{L^2_\tx([0,T]\times\T)}) \|\ga_0\|_{\FS^2}.
	\end{equation}
\end{proof}

We have the following corollary of Lemma \ref{lem:averaged Strichartz V}. 
\begin{corollary}\label{cor:Stri V renormalised}
	Let $T \in [0,2\pi]$.
	For any $V \in L^2_\tx([0,T]\times\T)$ and $f \in L^2_\tx([0,T]\times\T)$ such that $\lg f \rg = 0$, we have
	\begin{equation}
		\No{\int_0^T dt \, U_V^*(t) f(t) U_V(t) }_{\FS^2}
		\ls \Ph(T^{\theta}\|V\|_{L^2_\tx([0,T]\times\T)}) \|f\|_{L^2_\tx([0,T]\times\T)},
	\end{equation}
	where $\Ph:[0,\I) \to [0,\I)$ is a monotone increasing function.
\end{corollary}

Finally we prove the following estimate that is used in our contraction mapping arguments.
\begin{lemma}\label{lem:difference}
	For any $0\le T \le 2\pi$ and $V, W \in L^2_\tx([0,T]\times\T)$, we have
\begin{align}
	&\|\rh(U_V(t)\ga_0 U_V^*(t)) - \rh(U_W(t)\ga_0 U_W^*(t))\|_{L^2_\tx([0,T]\times\T)}  \\
	&\quad = \|\vrh(U_V(t)\ga_0 U_V^*(t)) - \vrh(U_W(t)\ga_0 U_W^*(t))\|_{L^2_\tx([0,T]\times\T)} \\
	&\quad \ls \Ph(T^{\theta}\|V\|_{L^2_\tx([0,T]\times\T)}+T^{\theta}\|W\|_{L^2_\tx([0,T]\times\T)}) T^{2\theta}  \|V - W\|_{L^2_\tx([0,T]\times\T)} \|\ga_0\|_{\FS^2},
\end{align}
where $\Ph:[0,\I) \to [0,\I)$ is a monotone increasing function.
\end{lemma}

\begin{proof}
	Let $\varrho_V:=\rh(U_V(t)\ga_0 U_V^*(t))$ and $\varrho_W:=\rh(U_W(t)\ga_0 U_W^*(t))$.
	Since $\lg \varrho_V(t) - \varrho_W(t) \rg_x \equiv 0$ for all $t\in[0,T]$, we have
	\begin{multline}
    \|\varrho_V - \varrho_W\|_{L^2_\tx([0,T]\times\T)}
	= \sup
	  \abs{\int_0^T \int_\T  dx \, dt \, f(t,x)(\varrho_V(t,x) - \varrho_W(t,x))} \\
	 \le  \sup
	            \abs{\int_0^T \int_\T  dx \,dt \, f(t,x)\rh((U_V(t) - U_W(t))\ga_0 U_V^*(t)) } \\
	+  \sup \abs{\int_0^T \int_\T  dx \, dt \, f(t,x)\rh(U_W(t) \ga_0 (U_V^*(t)-U_W^*(t)))} =: \SA+ \SB,
	\end{multline}
    where the suprema are taken over the set 
        \begin{equation*}
        \{f \in L^2_{[0,T]\times\T}: \|f\|_{L^2_\tx([0,T]\times\T)}\le 1 \text{ and } \lg f \rg = 0\}.
    \end{equation*}
    We only give details for bounding $\SA$ since a similar argument holds for $\SB$
	\footnote{When estimating $\SB$, instead of using \eqref{eq:UV UW difference}, we write
	\begin{equation}
		U_V(t)-U_W(t) = -(U_W(t)-U_V(t)) =  i \int_0^t U_W(t,\ta)(W(\ta)-V(\ta))U_V(\ta)d\ta.
		\end{equation}}.
	Note that 
	\begin{equation}\label{eq:UV UW difference}
		U_V(t)-U_W(t) =-i\int_0^t d\ta \, U_V(t,\ta)(V(\ta)-W(\ta))U_W(\ta).
	\end{equation}
	Let $f \in L^2_\tx([0,T]\times\T)$ such that $\lg f \rg = 0$. 
	Then by H\"older's inequality, we have
	\begin{multline}
		\abs{\int_0^T \int_\T dx \, dt \, f(t,x) \rh((U_V(t) - U_W(t))\ga_0 U_V^*(t))}\\
		 = \abs{\Tr\Dk{\int_0^T dt \, U_V^*(t) f(t) U_V(t) \int_0^t d\ta \, U_V(\ta) (V(\ta) - W(\ta)) U_W(\ta) \ga_0}}\\
		 \le \No{\int_0^T dt \, U_V^*(t) f(t) U_V(t) \int_0^t d\ta \, U_V(\ta) (V(\ta) - W(\ta)) U_W(\ta)}_{\FS^2} \|\ga_0\|_{\FS^2}.
	\end{multline}
	Set $g_0(\ta) := |V(\ta) - W(\ta)|^{1/2}$ and $g_0(\ta)g_1(\ta) = V(\ta)- W(\ta)$.
	Then
	\begin{multline}
		\No{\int_0^T dt \, U_V^*(t) f(t) U_V(t) \int_0^t d\ta \, U_V(\ta) (V(\ta) - W(\ta)) U_W(\ta)}_{\FS^2} \\
		 \le \No{\int_0^T dt \, U_V^*(t) f(t) U_V(t) \int_0^t d\ta  \, U_V(\ta) g_0(\ta)}_{\FS^2(L^2_{\tau,x} \to L^2_x)} 
		           \|g_1(\ta) U_W(\ta)\|_{L^2_x \to L^2_{\tau,x}}.
	\end{multline}
	By Lemma \ref{lem:Stri V} and Corollary \ref{cor:Stri V renormalised}, it follows that 
	\begin{multline}
		\No{\int_0^T dt \, U_V^*(t) f(t) U_V(t) \int_0^T d\ta \, U_V(\ta) g_0(\ta)}_{\FS^2(L^2_{\tau,x} \to L^2_x)} \\
		 \le \Ph(\|V\|_{L^2_\tx([0,T]\times\T)} T^{\theta})T^{\theta} \|f\|_{L^2_\tx([0,T]\times\T)} \|g_0\|_{L^4_\tx([0,T]\times\T)}.
	\end{multline}
	Using \cite[Corollary 4.9]{BHS25}, we have
	\begin{multline}
		\No{\int_0^T dt \, U_V^*(t) f(t) U_V(t) \int_0^t d\ta \, U_V(\ta) g_0(\ta)}_{\FS^2(L^2_{\tau,x} \to L^2_x)} \\
		 \le \Ph(\|V\|_{L^2_\tx([0,T]\times\T)} T^{\theta})T^{\theta} \|f\|_{L^2_\tx([0,T]\times\T)} \|g_0\|_{L^4_\tx([0,T]\times\T)}.
	\end{multline}
	By Lemma \ref{lem:Stri V}, we have
	\begin{equation}
		\|g_1(\ta) U_W(\ta)\|_{L^2_x \to L^2_{\tau,x}} \le \Ph(\|W\|_{L^2_\tx([0,T]\times\T)}T^{\theta}) T^{\theta} \|g_1\|_{L^4_\tx([0,T]\times\T)}.
	\end{equation}
	It follows that 
	\begin{multline}
		\abs{\int_0^T \int_\T dx\, dt \, f(t,x) \rh((U_V(t) - U_W(t))\ga_0 U_V^*(t))}\\
		 \le  \Ph(T^{\theta}\|V\|_{L^2_\tx([0,T]\times\T)} + T^{\theta}\|W\|_{L^2_\tx([0,T]\times\T)}) T^{2\theta} \|f\|_{L^2_\tx([0,T]\times\T)} \|V - W\|_{L^2_\tx([0,T]\times\T)}.
	\end{multline}
\end{proof}

\subsection{Proof of Propositions \ref{prop:cubic NLSS wellposed} and \ref{prop:cubic NLSS renormalised wellposed}}
\label{SEC:WP:proof}
In this section, we prove Propositions \ref{prop:cubic NLSS wellposed} and \ref{prop:cubic NLSS renormalised wellposed}.
Before proceeding, we need to clarify our notions of solution.
\begin{definition}[Solution to \eqref{eq:NLSS}]\label{def:solution}
	Let $\ga_0 \in \FS^1$ be self-adjoint.
	We call $\ga(t) \in C([0,T];\FS^1)$ a solution to
	\begin{equation}\label{eq:LS}
		\begin{dcases}
			i\pl_t \ga = [- \De  \pm \rh_\ga, \ga], \quad \ga:[0,T]\to \FS^1,& \\
			\ga(0)=\ga_0&
		\end{dcases}
	\end{equation}
	if $\ga(t)\in C([0,T];\FS^1)$ and $\varrho \in L^2_t([0,T];L^2_x(\T))$ satisfy
	\begin{equation}\label{eq:LS Duhamel}
		\begin{dcases}
			\ga(t)= U_{\pm \varrho} (t) \ga_0 U_{\pm \varrho}^* (t), \\
			\varrho = \rho_\ga.
		\end{dcases}	
	\end{equation}
\end{definition}
\begin{remark}
We note that since $\varrho \in L^2_t([0,T];L^2_x(\T))$ is real-valued, we have that the unitary propagator $U_{\pm \varrho}(t)$ is well-defined. Moreover, since $\varrho \in L^2_t([0,T];L^2_x(\T))$, Corollary \ref{cor:Stri V} ensures that $\rh_\ga = \rh(U_{\pm \varrho}(t)\ga_0 U_{\pm \varrho}^*(t))$ is well-defined in $L^2_t([0,T];L^2_x(\T))$.
\end{remark}
\begin{remark}
Let us remark that Definition \ref{def:solution} essentially corresponds to the von Neumann evolution for the density matrix $\gamma$. Another natural notion of solution would be the corresponding Duhamel formulation of the problem, given by
\begin{equation}\label{eq:Duhamel}
	\ga(t) = U(t)\ga_0 U^*(t) - i\int_0^t d\ta \, U(t-\ta)[\pm \rh_\ga(\ta), \ga(\ta)]U(\ta-t). 
\end{equation}
If $\ga(t) \in C([0,T];\FS^2)$ and $\rh_\ga \in L^2_\tx([0,T]\times\T)$, then both the left and right hand sides of \eqref{eq:Duhamel} make sense in $\CH^{-1}$, where
$$\|A\|_{\CH^{-1}} := \|\sd^{-1} A \sd^{-1}\|_{\FS^2}.$$
Here $\langle \cdot \rangle$ denotes the Japanese bracket.
Indeed, the left hand side and the first term of the right hand side of \eqref{eq:Duhamel} are well-defined in $\FS^2$, and in particular, make sense in $\CH^{-1}$.
Moreover, we have
\begin{equation*}
    \No{\int_0^t U(t-\ta)[\pm \rh_\ga(\ta), \ga(\ta)]U(\ta-t)d\ta}_{\CH^{-1}}
    \ls \int_0^t d\ta \No{\sd^{-1}\rh_\ga(\ta)}_{\CB(L^2(\T))} \|\ga(\ta)\|_{\FS^2}.
\end{equation*}
Bounding the operator norm with the Hilbert-Schmidt norm, one has
\begin{equation*}
    \|\sd^{-1} \rh_\ga(\ta)\|_{\CB(L^2(\T))} \le  \|\lg n \rg^{-1} e^{-inx} \rho_\ga(\ta,x)\|_{L^2_x \ell^2_n(\T \times \Z)} \ls \|\rho_\ga(\ta)\|_{L^2_x(\T)}.
\end{equation*}
Therefore
\begin{equation*}
    \No{\int_0^t U(t-\ta)[\pm \rh_\ga(\ta), \ga(\ta)]U(\ta-t)d\ta}_{\CH^{-1}}
    \ls t^{1/2} \|\rh_\ga\|_{L^2_\tx([0,T]\times\T)} \|\ga(\ta)\|_{C([0,T];\FS^2)} < \I.
\end{equation*}
One can show that every solution defined in Definition \ref{def:solution} satisfies the Duhamel formulation \eqref{eq:Duhamel} in $\CH^{-1}$.
Indeed, the identity
$$U_{\pm\varrho}(t) = U(t) \mp i \int_0^t U(t-\ta)\varrho(\ta) U_{\pm \varrho}(\ta)d\ta = U(t) + D(t),$$
implies
$$\ga(t)\equiv U_{\pm\varrho}(t) \ga_0 U_{\pm\varrho}^*(t) = U(t)\ga_0 U^*(t) + U(t)\ga_0 D^*(t) + D(t)\ga_0 U^*(t) + D(t)\ga_0 D^*(t).$$
An elementary but long calculation shows
$$
U(t)\ga_0 D^*(t) + D(t)\ga_0 U^*(t) + D(t)\ga_0 D^*(t)
= - i\int_0^t d\ta U(t-\ta)[\pm \rh_\ga(\ta), \ga(\ta)]U(\ta-t).
$$
However, it is not clear that the converse is true. We leave this as a question for future research.
\end{remark}
Similarly, we define the solution to the renormalised cubic NLS system.
\begin{definition}[Solution to \eqref{eq:NLSS'}]\label{def:solution renormalised}
	Let $\ga_0 \in \FS^2$ be self-adjoint.
	We call $\ga(t) \in C([0,T];\FS^2)$ a solution to
	\begin{equation}\label{eq:LS normalised}
		\begin{dcases}
			i\pl_t \ga = [-\De \pm \vrh_\ga, \ga], \quad \ga:[0,T]\to \FS^2,& \\
			\ga(0)=\ga_0&
		\end{dcases}
	\end{equation}
	if $\ga(t)\in C([0,T];\FS^{2})$ and $\varrho \in L^2_t([0,T];L^2_x(\T))$ satisfy
	\begin{equation}\label{eq:LS Duhamel renormalised}
		\begin{dcases}
			\ga(t)= U_{\pm \varrho} (t) \ga_0 U_{\pm \varrho}^* (t), \\
			\varrho = \vrh_\ga.
		\end{dcases}	
	\end{equation}
\end{definition}
We now give the proof of Propositions \ref{prop:cubic NLSS wellposed} and \ref{prop:cubic NLSS renormalised wellposed}.
\begin{proof}[Proof of Proposition \ref{prop:cubic NLSS wellposed}]
	We first prove the well-posedness portion of the proposition. For $T>0$, define $\Ga:L^2_\tx([0,T]\times\T) \to L^2_\tx([0,T]\times\T)$ by
	$$\Ga[\varrho]:= \rh(U_{\pm \varrho}(t) \ga_0 U_{\pm \varrho}^*(t)).$$
	Let
    \begin{equation*}
    E(T,R):=\Ck{\varrho : \|\varrho\|_{L^2_\tx([0,T]\times\T)} \le R},
    \end{equation*}
	where $T$ will be chosen sufficiently small later. For any $\varrho \in E(T,R)$, Corollary \ref{cor:Stri V} implies
	\begin{equation}
		\|\Ga[\varrho]\|_{L^2_\tx([0,T]\times\T)}
		\le \Ph(\|\varrho\|_{L^2_\tx([0,T]\times\T)}T^{\theta})T^{2\theta}\|\ga_0\|_{\FS^1}
		\le \Ph(RT^{\theta})T^{2\theta}\|\ga_0\|_{\FS^1} \le R
	\end{equation}
	if we take $T=T(R,\|\ga_0\|_{\FS^1})$ small enough.
	Therefore $\Ga:E(T,R) \to E(T,R)$ is well-defined. 
	For any $\varrho_1, \varrho_2 \in E(T,R)$, Lemma \ref{lem:difference} implies that
	\begin{align}
		\|\Ga[\varrho_1] - \Ga[\varrho_2]\|_{L^2_\tx([0,T]\times\T)}
		&\le \Ph(T^{\theta} R) T^{2\theta} \|\varrho_1-\varrho_2\|_{L^2_\tx([0,T]\times\T)} \|\ga_0\|_{\FS^2} \\
		&\le \tw \|\varrho_1-\varrho_2\|_{L^2_\tx([0,T]\times\T)}
	\end{align}
	if we choose sufficiently small $T=T(R,\|\ga_0\|_{\FS^1})>0$.
	So by the contraction mapping theorem there is at least one solution $\varrho \in E(T,R)$.
	Now we define $\ga(t):=U_{\pm \varrho}(t)\ga_0 U_{\pm \varrho}^*(t)$.
	Then, $(\ga(t),\varrho)$ is a solution to \eqref{eq:NLSS}.
	We can prove uniqueness of the solution, continuity of the data-to-solution map, and the blow up alternative in the standard way.

We now turn our attention to the ill-posedness portion of the proposition.
Recalling the $e_n$ defined in \eqref{en_defintion}, set
	\begin{equation}
		\ga_N:= \frac{1}{N^{1/\al + \ep}}\sum_{n=1}^N |e_n \rg \lg e_n|
	\end{equation}
	for sufficiently small $\ep>0$.
    A direct calculation shows that $\ga_N$ is a stationary solution to \eqref{eq:NLSS} because the density is a constant.
    We have
    \begin{equation*}
        \|\ga_N\|_{\FS^\al} = \frac{1}{N^\ep} \to 0 \quad \text{as } N \to \I.
    \end{equation*}
	However we also have
	\begin{equation}
		\|\CS_2(\ga_N)\|_{L^2_\tx([0,T]\times\T)} = \|\rh_{\ga_N}\|_{L^2_\tx([0,T]\times\T)}  \sim N^{1-1/\al-\ep} \to \I \quad \text{as } N \to \I,
	\end{equation}
	where we chose small $\ep>0$ such that $1-1/\al -\ep>0$.
    Therefore, there is no continuous extension $\CS_2:E_R^{\al} \to L^2_\tx([0,T]\times\T)$ for $\al>1$.
\end{proof}

\begin{proof}[Proof of Proposition \ref{prop:cubic NLSS renormalised wellposed}]
We omit the proof of well-posedness part because it is similar to the proof of Proposition \ref{prop:cubic NLSS wellposed}, except that Lemma \ref{lem:averaged Strichartz V} means that $\Gamma$ is now well-defined for $\gamma_0 \in \FS^2$.

We now prove the ill-posedness. Suppose that we have a continuous extension of the data-to-density map $\mathcal{S}_2:E_R^\al \to L^2_\tx([0,T]\times\T)$ for $\al> 2$. Recall the functions $e_n$ from \eqref{en_defintion}. Set $\phi_n = 2^{-1/2}(e_n(x) + e_{-n}(x))$.
Define
\begin{equation}
	\ga_{0,N} := \frac{1}{N^{1/2}}\sum_{n=1}^N |\phi_n\rg \lg \phi_n|.
\end{equation}
By assumption, for each $N \in \BN$, there exists $\varrho_N \in L^2_\tx([0,T]\times\T)$ such that
\begin{equation}
	\varrho_N =\vrh(U_{\pm \varrho_N}(t) \ga_{0,N} U_{\pm \varrho_N}^*(t)).
\end{equation}
Moreover, since $\|\ga_{0,N}\|_{\FS^\al} \to 0$ as $N\to \I$ because of $\al > 2$ and $\CS_2 : E_R^\al \to L^2_\tx([0,T]\times\T)$ is continuous, we have 
$$
\|\varrho_N\|_{L^2_\tx([0,T]\times\T)} \to 0 \quad \text{as } N \to\I.
$$
Therefore, we obtain
\begin{align}
	\|\vrh(U(t)\ga_{0,N} U^*(t))\|_{L^2_\tx([0,T]\times\T)}
		&\le \|\varrho_N\|_{L^2_\tx([0,T]\times\T)}\\
		&\quad  + \|\vrh(U_{\pm\varrho_N}(t)\ga_{0,N} U_{\pm\varrho_N}^*(t)) - \vrh(U(t)\ga_{0,N} U^*(t))\|_{L^2_\tx([0,T]\times\T)}.
\end{align}
By the proof of Theorem \ref{thm:L2 bound renormalised}, we have
\begin{equation}
	\|\vrh(U(t)\ga_{0,N} U^*(t))\|_{L^2_\tx([0,T]\times\T)} \sim 1.
\end{equation}
By Lemma \ref{lem:difference}, we have
\begin{multline*}
 	\|\vrh(U_{\pm\varrho_N}(t)\ga_{0,N} U_{\pm\varrho_N}^*(t)) - \vrh(U(t)\ga_{0,N} U^*(t))\|_{L^2_\tx([0,T]\times\T)} \\
	\quad \le \Ph(T^{\theta}\|\varrho_N\|_{L^2_\tx([0,T]\times\T)}) T^{2\theta} \|\varrho_N\|_{L^2_\tx([0,T]\times\T)} \|\ga_{0,N}\|_{\FS^2} \\
	\quad = \Ph(T^{\theta}\|\varrho_N\|_{L^2_\tx([0,T]\times\T)}) \|\varrho_N\|_{L^2_\tx([0,T]\times\T)} T^{2\theta}.   
\end{multline*}
Hence
\begin{equation}
	1
	\ls \|\varrho_N \|_{L^2_\tx([0,T]\times\T)}
	+ \Ph(T^{\theta}\|\varrho_N\|_{L^2_\tx([0,T]\times\T)}) \|\varrho_N\|_{L^2_\tx([0,T]\times\T)} T^{2\theta}.
\end{equation}
Taking $N \to \I$, it follows that
\begin{equation}
	1 \ls 0,
\end{equation}
which is a contradiction. Therefore no continuous extension exists.
\end{proof}

\section{Renormalisation on \texorpdfstring{$\T^d$}{Td} with \texorpdfstring{$d\ge 2$}{d2}}\label{sec:higher dimensional case}
Everything we have discussed so far has been for the one dimensional torus. In this section, we consider the case of $\T^d$ for $d \ge 2$. Set $p^* := \frac{d+2}{d}$ and write
\begin{align}
	&e_n(x) := \frac{1}{(2\pi)^{d/2}}e^{in\cdot x}, \quad \wh{\phi}(n) = \int_{\T^d} \ov{e_n(x)} \phi(x)dx, \\
	&U(t)\phi(x) = \sum_{n \in \Z^d} e_n(x)e^{-it|n|^2} \wh{\phi}(n), \quad P_{\le N} \phi (x):= \sum_{|n_1|, \dots, |n_d|\le N} e_n(x)\wh{\phi}(n).
\end{align}

In \cite{Nak20}, Nakamura proved the following theorem.
\begin{theorem}[Theorem 1.4 in \cite{Nak20}] \label{thm:Nakamura full version}
	Let $d \ge 1$ and $\si \in (0,1/p^*]$.
	Then, if $1/\al > 1 - \si/d$, we have
	\begin{equation}\label{eq:Nakamura general}
		\|\rh(U(t)P_{\le N} \ga_0 P_{\le N} U^*(t))\|_{L^{p^*}_{t,x}(\T^{1+d})} \ls_\si N^\si \|\ga_0\|_{\FS^\al}.
	\end{equation}
	This is almost sharp in the sense that \eqref{eq:Nakamura general} fails if $1/\al < 1 - \si/d$.
\end{theorem}
The $d$-dimensional analogue of the renormalised density is given by
\begin{equation}
	\vrh_A(x) = \rh_A(x) - \frac{1}{(2\pi)^d} \int_{\T^d} dy \, \rh_A(y) = \rh_A(x) - \frac{1}{(2\pi)^d} \Tr A.  
\end{equation}
It is thus natural to consider an estimate in the form of
\begin{equation}\label{eq:general renormalised}
	\|\vrh(U(t)P_{\le N}\ga_0 P_{\le N} U^*(t))\|_{L^{p^*}_{t,x}(\T^{1+d})} \ls_\si N^\si \|\ga_0\|_{\FS^\al}.
\end{equation}
In this section, we first give a proof of Theorem \ref{thm:Nakamura full version} in an alternative way to \cite{Nak20}. We then give a necessary condition on $(\si,\al)$ for \eqref{eq:general renormalised} to hold.

\subsection{Alternative proof of \texorpdfstring{\cite[Theorem 1.4]{Nak20}}{Theorem 1.4}}
We first prove the following lemma.

\begin{lemma}\label{lem:short time dispersive estimate}
	Let $d \ge 1$, $N \in \BN$.
	Then, we have
	\begin{equation}
		\|M_f U(t)P_{\le N} M_g\|_{\FS^p} \ls |t|^{-d/p} \|f\|_{L^p(\T^d)} \|g\|_{L^p(\T^d)} \quad \text{if } |t|\le \frac{1}{N}
	\end{equation}
	for any $f, g \in L^p(\T^d)$ and $p \in [2,\I]$, where $M_F$ is the multiplication operator by $F:\T^d \to \C$. 
\end{lemma}

\begin{proof}
	Let us recall from  \cite[Equation (3.1)]{Nak20} or \cite[Equation (5.9)]{KPV91} the short-time dispersive estimate given by
	\begin{equation}
		\|U(t)P_{\le N}(x,x')\|_{L^\I_{x,x'}(\T^{d+d})} \ls \frac{1}{|t|^{d/2}} \quad \text{for } |t|\le \frac{1}{N},
	\end{equation}
	where $U(t)P_{\le N}(x,x')$ is the integral kernel of $U(t)P_{\le N}:L^2(\T^d) \to L^2(\T^d)$.
	It follows that
	\begin{equation}\label{eq:bilinear S2}
    \begin{aligned}
    \|M_f U(t)P_{\le N} M_g\|_{\FS^2}
		&= \|f(x)\Dk{U(t)P_{\le N}}(x,x')g(x')\|_{L^2_{x,x'}(\T^{d+d})} \\
		&\ls |t|^{-d/2} \|f\|_{L^2(\T^d)} \|g\|_{L^2(\T^d)},
    \end{aligned}
	\end{equation}
	where we used Lemma \ref{lem:HS}.
	Interpolating between \eqref{eq:bilinear S2} and
	\begin{equation}\label{eq:bilinear B}
		\|M_f U(t)P_{\le N} M_g\|_{\CB(L^2(\T^d))} \le \|f\|_{L^\I(\T^d)} \|g\|_{L^\I(\T^d)}
	\end{equation}
	 completes the proof.
\end{proof}

Next, we give a short-time estimate.
\begin{lemma}\label{lem:short time}
	Let $d \ge 1$, $N \in \N$, and $I \subset [0,2\pi]$ be an interval such that $|I|\le 1/N$.
	Then, we have
	\begin{equation}\label{eq:short time}
		\No{\rh(U(t)P_{\le N}\gamma_0 P_{\le N}U^*(t))}_{L^{p^*}_\tx(I\times \T^d)} \ls \|\ga_0\|_{\FS^{\al^*}}, 
	\end{equation}
	where $\al^* = \frac{d+2}{d+1}$.
\end{lemma}

\begin{remark}
	In \cite{Nak20}, Nakamura proved the equivalent estimate by the duality argument introduced in \cite{FS17}. However, here we show we can prove Lemma \ref{lem:short time} using the same argument as \cite[Proof of Theorems 1 and 2]{FLLS24}, which is an older paper than \cite{FS17}.
	The proof below has no originality; however we think that it is meaningful to point out that we can apply the method in \cite{FLLS24} to our problem.
\end{remark}

\begin{proof}
By a Schatten duality argument, \eqref{eq:short time} is equivalent to
\begin{equation}
	\No{P_{\le N}\int_I dt U^*(t)V(t) U(t) P_{\le N}}_{\FS^{d+2}} \ls \|V\|_{L^{\frac{d+2}{2}}_\tx(I\times\T^d)}.
\end{equation}
We can assume without loss of generality that $V(t,x)\ge 0$ for all $(t,x)\in I \times \T^d$.
We have 
\begin{align}
	&\No{P_{\le N}\int_I dt U^*(t)V(t) U(t) dt P_{\le N}}_{\FS^{d+2}}^{d+2}
	= \Tr\Dk{\K{P_{\le N} \int_I dt U^*(t) V(t) U(t)P_{\le N} }^{d+2}} \\
	&\quad = \int_I dt_1 \int_Idt_2 \cdots \int_I d t_{d+2} \\
    &\qquad \times \Tr\K{\sqrt{V(t_1)} P_{\le N} U(t_1 - t_2) \sqrt{V(t_2)} \sqrt{V(t_2)} U(t_2 - t_3) \cdots U(t_{d+2} - t_1)\sqrt{V(t_1)} } \\
	&\quad \le \int_I dt_1 \int_Idt_2 \cdots \int_I d t_{d+2} \No{\sqrt{V(t_1)} P_{\le N} U(t_1 - t_2) \sqrt{V(t_2)}}_{\FS^{d+2}} \cdots \\
	&\qquad \cdots\No{\sqrt{V(t_{d+2})} P_{\le N} U(t_{d+2} - t_1) \sqrt{V(t_1)}}_{\FS^{d+2}}.
\end{align}
Since $|t_j - t_{j+1}| \le 1/N$, Lemma \ref{lem:short time dispersive estimate} with $p = d+2$ implies
\begin{multline}
\No{P_{\le N}\int_I dt U^*(t)V(t) U(t) dt P_{\le N}}_{\FS^{d+2}}^{d+2} \\
	\ls \int_I dt_1 \int_Idt_2 \cdots \int_I d t_{d+2}
	       \frac{\|V(t_1)\|_{L^{\frac{d+2}{2}}(\T^d)} \cdots \|V(t_{d+2})\|_{L^{\frac{d+2}{2}}(\T^{d})}}{|t_1 - t_2|^{\frac{d}{d+2}}\cdots |t_{d+1}-t_{d+2}|^{\frac{d}{d+2}} |t_{d+2} - t_1|^{\frac{d}{d+2}}}.
\end{multline}
The result follows from the multilinear Hardy--Littlewood--Sobolev inequality, see \cite[Theorem 4]{FLLS24}.
\end{proof}

We now prove Theorem \ref{thm:Nakamura full version}.
\begin{proof}
	It suffices to prove
	\begin{equation}\label{eq:Td Nsigma}
		\No{\rh(U(t)P_{\le N}\ga_0 P_{\le N}U^*(t))}_{L^{p*}_\tx(\T^{1+d})} \ls N^{1/p^*} \|\ga_0\|_{\FS^{\al^*}}.
	\end{equation}
    Indeed, using the one-body estimate from \cite{Bou93, BD13}
	\begin{equation}
		\No{|U(t)P_{\le N} \phi|^2}_{L^{p^*}_\tx(\T^{1+d})} \ls_\ep N^\ep \|\phi\|_{L^2_x(\T^d)}^2,
	\end{equation}
	and the triangle inequality, we have 
\begin{equation}\label{eq:Td Nepsilon}
	\No{\rh(U(t)P_{\le N}\ga_0 P_{\le N} U^*(t))}_{L^{p*}_\tx(\T^{1+d})} \ls_\ep N^\ep \|\ga_0\|_{\FS^1}.
\end{equation}
Interpolating between \eqref{eq:Td Nsigma} and \eqref{eq:Td Nepsilon}  yields \eqref{eq:Nakamura general}. To show \eqref{eq:Td Nsigma}, we decompose $T^* = \sqcup_{j\in J} I_j$, where $|I_j|\le 1/N$ and $\# J \ls N$.
Therefore by Lemma \ref{lem:short time}, we conclude that 
\begin{align}
	\No{\rh(U(t)P_{\le N}\ga_0 P_{\le N} U^*(t))}_{L^{p*}_\tx(\T^{1+d})}
	&\ls \bigg(\sum_{j\in J}\No{\rh(U(t)P_{\le N}\ga_0 P_{\le N}U^*(t))}_{L^{p*}_\tx(I_j\times\T^{d})}^{p^*}\bigg)^{1/p*} \\
	&\ls (\# J)^{1/p^*} \|\ga_0\|_{\FS^{\al^*}}\ls N^{1/p*} \|\ga_0\|_{\FS^{\al^*}}.
\end{align}
\end{proof}

\subsection{A ``negative'' result on higher dimensional toruses}
In this section, we show a strong necessary condition on $(\al,\si)$ for the orthonormal Strichartz estimates satisfied by the renormalised density. In particular, we prove the following result.
\begin{proposition}
	Let $d \ge 2$, $\si \in (0,1/p^*]$, and $N \in \BN$.
	Suppose that
	\begin{equation}
		\No{\vrh(U(t)P_{\le N} \ga_0 P_{\le N}U^*(t))}_{L^{p^*}(\T^{1+d})} \ls_\si N^\si \|\ga_0\|_{\FS^\al}
	\end{equation}
	for all $\ga_0 \in\FS^\al$. Then $1/\al \ge 1 - \frac{\si}{d-1}$.
\end{proposition}

\begin{remark}
	It is clear that $1/\al \ge 1- \frac{\si}{d-1}$ is (slightly) milder than $1/\al \ge 1-\si/d$. However, we need $\al \le 1 + \ep$ with small $\ep>0$ if $\si$ is small. Therefore the renormalised density has almost no gain of Schatten exponents when $\si$ is close to zero. This means that when $d \geq 2$, the situation is similar to the non-renormalised case.
\end{remark}

\begin{proof}
	Set $k = (1,0,\dots,0) \in \Z^d$.
	Let
	$$
	\CA_N := \{n = (0,n_2,\dots,n_d) \in \Z^d: |n_2|,\dots,|n_d|\le N\}.
	$$
	Then we have $\# \CA_N \sim N^{d-1}$.
	For $n \in \CA_N$, define $\phi_n(x)$ by
	\begin{align}
		\phi_n(x) :=
			\frac{e_{n+k}(x) + e_{n-k}(x) }{\sqrt{2}}.
	\end{align}
	Then $(\phi_n)_{n\in \CA_N}$ is an orthonormal system in $L^2(\T^d)$.
	Since
	\begin{equation}
		U(t)\phi_n(x) = \frac{e^{-it(|n|^2 + 1)}(e_{n+k} + e_{n-k}) }{\sqrt{2}},
	\end{equation}
	we have
	\begin{align}
		\sum_{n\in\CA_N} |U(t)\phi_n|^2 - \frac{1}{(2\pi)^d} \sum_{n \in \CA_N} 1= \frac{1}{(2\pi)^d} \sum_{n \in \CA_N} \cos(2x_1) \sim N^{d-1} \cos(2x_1).
	\end{align}
	Therefore we obtain 
	\begin{equation}
		\No{\sum_{n\in\CA_N} |U(t)\phi_n|^2 - \frac{1}{(2\pi)^d} \sum_{n \in \CA_N} 1}_{L^{p^*}_\tx(\T^{1+d})}
		\gs N^{d-1}.
	\end{equation}
	If we have \eqref{eq:general renormalised}, we need
	\begin{equation}
		N^{d-1} \ls \No{\sum_{n\in\CA_N} |U(t)\phi_n|^2 - \frac{1}{(2\pi)^d} \sum_{n \in \CA_N} 1}_{L^{p^*}_\tx(\T^{1+d})}
		\ls N^\si \K{\# \CA_N}^{1/\al} \sim N^{\si + (d -1)/\al},
	\end{equation}
	which forces $1/\al \ge 1 - \frac{\si}{d-1}$.
\end{proof}

\appendix
\section{Construction of \texorpdfstring{$U_V(t,s)$}{UV(t,s)}}\label{sec:appendix}
\label{SEC:uv:construction}
In this section, we briefly outline the construction of the propagator $U_V(t,s)$ for real-valued $V \in L^2_{t,x}([0,T]\times\T)$. Recall that we consider the linear Schr\"odinger equation given by
\begin{equation}
\label{linear_potential_schro}
\begin{cases}
    i \partial_t u + \De u = V(t,x)u, \quad u:[0,T] \times \T \to \C,\\
    u_{t = s}(x) = u_{0}(x) \in L^2(\T).
\end{cases}
\end{equation}
We want to define
\begin{equation*}
U_V(t,s) u_0 := u(t),
\end{equation*}
where $u(t)$ denotes the solution of \eqref{linear_potential_schro}.
To do this, we first need to prove that there is always a global solution to \eqref{linear_potential_schro}. Note that for the existence of local solutions, we do not require that $V$ be real-valued.
\begin{lemma}
\label{linear_well_posedness}
Suppose that $V \in L^2_{t,x}([0,T]\times\T)$, $u_0 \in L^2(\T)$, and $s \in [0,T]$. Then there is some $\tau = \tau(\|V\|_{L^2_{t,x}},\|u_0\|_{L^2})$ and a unique $u \in C([s-\tau,s+\tau];L^2(\T))$ that solves \eqref{linear_potential_schro}.
\end{lemma}
\begin{remark}
If $s = 0$ or $T$, we only consider the solution on the interval  $I= [0,\tau]$ or $[T-\tau,T]$, which ensures that $V$ is well-defined on $I$.
\end{remark}
\begin{proof}
Fix $I= [s-\tau,s+\tau]$ with $\tau$ sufficiently small to be determined later. We make a contraction on the set
\begin{equation*}
E := \{u \in X^{0,b}_I : \|u\|_{X^{0,b}_I} \leq 2\|u_0\|_{L^2}\},
\end{equation*}
where $b = \frac{1}{2} + \eta$. Define the map $\Gamma : E \to E$ via
\begin{equation*}
\Ga u (t) := U(t-s) u_0 - i \int_s^t dt' \, U(t-t') V(t') u(t') .
\end{equation*}
Recalling the standard properties of $X^{0,b}_I$ spaces from \cite{ET16} and using \eqref{eq:L4}, H\"older's inequality, and \eqref{eq:Stri inhomogeneous}, we have
\begin{equation*}
\|\Gamma u\|_{X^{0,b}_I} \leq \|u_0\|_{L^2} + (2\tau)^\theta \|V\|_{L^2_{t,x}} \|u\|_{L^4_{t,x}} \leq \|u_0\|_{L^2} + (2\tau)^\theta \|V\|_{L^2_{t,x}} \|u\|_{X^{0,b}_I}.
\end{equation*}
Taking $\tau = \tau(\|V\|_{L^2_{t,x}}, \|u_0\|_{L^2})$ sufficiently small, we get that $\Gamma$ is well-defined. Similarly, we get that $\Gamma$ is a contraction for $\tau$ sufficiently small.
\end{proof}
To define $U_V$ for arbitrary times, we need to extend local well-posedness to global well-posedness. To do this, we need to prove conservation of mass.
\begin{lemma}
\label{linear_mass_conservation_lemma}
Suppose that $V \in L^2_{t,x}([0,T]\times\T)$ is real-valued and $u_0 \in L^2(\T)$. Then the solution to \eqref{linear_potential_schro} satisfies
\begin{equation*}
\|u(t)\|_{L^2(\T)}^2 = \|u_0\|_{L^2}^2.
\end{equation*}
\end{lemma}
\begin{proof}
We consider the approximation $V^N \to V$ in $L^2_{t,x}$, where we take $V^N$ to be real-valued smooth functions. Then solutions to 
\begin{equation*}
    i \partial_t u^N + \De u^N = V^N(t,x)u^N
\end{equation*}
display persistence of regularity for smooth initial data. It follows that $\|u^N(t)\|_{L^2}^2 = \|u_0\|_{L^2}^2$. Here we have used that $V^N$ is real-valued. Arguing as in the proof of Lemma \ref{linear_well_posedness}, it follows that
\begin{equation*}
\|u-u^N\|_{X^{0,b}_I} \lesssim (2\tau)^\theta \|V - V^N\|_{L^2_{t,x}} \|u^N\|_{X^{0,b}_I} + (2\tau)^\theta \|V\|_{L^2_{t,x}} \|u - u^N\|_{X^{0,b}_I}.
\end{equation*}
So locally, one has that $\|u-u^N\|_{X^{0,b}_I} \to 0$. However, one has
\begin{equation*}
|\|u(t)\|_{L^2} - \|u_0\|_{L^2}| = |\|u(t)\|_{L^2} - \|u^N(t)\|_{L^2}| \leq \|u(t) - u^N(t)\| \leq \|u - u^N\|_{L^\infty_I L^2} \lesssim \|u-u^N\|_{X^{0,b}_I}.
\end{equation*}
Therefore one has conservation of mass for the original equation.
\end{proof}
We have the following corollary.
\begin{corollary}
\label{linear_global_wellposedness}
Suppose that $V \in L^2_{t,x}([0,T]\times \T)$ is real-valued, $u_0 \in L^2(\T)$, and $s \in [0,T]$. Then there is a unique solution to \eqref{linear_potential_schro} in $C_t([0,T];L^2(\T))$ which conserves mass.
\end{corollary}
We are now able to complete the construction of the propagator $U_V(t,s)$.
\begin{proposition}
\label{UV_propositon}
Suppose that $V \in L^2_{t,x}([0,T]\times\T)$ is real-valued. Define the linear map $U_V(t,s): L^2(\T) \to L^2(\T)$ by $U_V(t,s) u_0 := u(t)$, where $u$ is the solution of 
\begin{equation*}
\begin{cases}
    i \partial_t u + \Delta u = V(t,x) u, \\
    u(s,x) = u_0(x) \in L^2(\T).
\end{cases}
\end{equation*}
Then $U_V(t,s)$ is defined for all $t,s \in [0,T]$, and we have the following properties.
\begin{itemize}
    \item[(i)] $U_V(t,s) : L^2 \to L^2$ is unitary and $(U_V(t,s))^* = U_V(s,t)$.
    \item[(ii)] For any $t,s,r \in [0,T]$, one has $U_V(t,s) = U_V(t,r)U_V(r,s)$.
     \item[(iii)] Suppose that $(t_n,s_n) \to (t,s)$ and $u_0 \in L^2(\T)$. Then $U_V(t_n,s_n)u_0 \to U_V(t,s)u_0$ in $L^2(\T)$.
\end{itemize}
\end{proposition}
\begin{proof}
The map is well-defined by global well-posedness, and it is linear because \eqref{linear_potential_schro} is linear. For property $(i)$, we note that conservation of mass means $U_V(t,s)$ is an isometry on $L^2(\T)$, and backwards well-posedness ensures that the map is surjective. Recall that every surjective linear isometry on a Hilbert space is unitary. Therefore, because $(U_V(t,s))^{-1} =  U_V(s,t)$, it follows that $(U_V(t,s))^* = U_V(s,t)$. Property $(ii)$ follows from uniqueness of solutions. For property $(iii)$, we note that we can write
\begin{multline*}
(U_V(t,s) - U_V(t_n,s_n))u_0 = (U_V(t,s) - U_V(t_n,s))u_0 + (U_V(t_n,s) - U_V(t_n,s_n))u_0 \\
= (U_V(t,s) - U_V(t_n,s))u_0 + U_V(t_n,t^*) (U_V(t^*,s) - U_V(t^*,s_n))u_0,
\end{multline*}
for some $t^* \in [0,T]$. In particular, by unitarity it suffices to show that
\begin{equation}
\label{convergence_result}
(U_V(t_n,s) - U_V(t,s))u_0 \to 0 
\end{equation}
in $L^2(\T)$ as $t_n \to t$. We write the left hand side of \eqref{convergence_result} in terms of Duhamel's formula, which yields
\begin{multline}
\label{Duhamel_1}
(U_V(t_n,s) - U_V(t,s))u_0 = (U(t_n - s) - U(t-s))u_0 \\
- i \int_s^{t_n} dt' \, U(t_n-t') V(t') u(t')
+ i \int_s^{t} dt' \, U(t-t') V(t') u(t').
\end{multline}
The strong continuity of $U(t)$ implies that the first term on the right hand side converges to zero in $L^2$ as $t_n \to t$. For the second term, we rewrite it as
\begin{equation}
\label{Duhamel_2}
i \int_{t}^{t_n} dt' \,  U(t_n-t') V(t') u(t')  + i \int_s^{t} dt' \, (U(t-t') - U(t_n-t')) V(t') u(t').
\end{equation}
For the first term in \eqref{Duhamel_2}, we bound
\begin{equation*}
\left\|\int_{t}^{t_n} dt' \,  U(t_n-t') V(t') u(t')\right\|_{L^2(\T)}
\lesssim |t-t_n|^{\theta} \|V\|_{L^2_{t,x}} \|u\|_{L^{4/3}_{t,x}} \to 0
\end{equation*}
as $t_n \to t$ by Lemma \ref{lem:Stri inhomogeneous}. For the second term in \eqref{Duhamel_2}, we write
\begin{multline*}
\left\|\int_s^{t} dt' \, (U(t-t') - U(t_n-t')) V(t') u(t')\right\|_{L^2} = \left\|(U(t) - U(t_n)) \left(\int_s^{t} dt' \, U(-t') V(t') u(t')\right)\right\|_{L^2} \\
= \left\|(U(t) - U(t_n)) \Psi(t,s)\right\|_{L^2}.
\end{multline*}
Arguing as in the proof of \eqref{eq:Stri V dual}, it follows that $\Psi(t,s) \in L^2_x(\T)$, and the result follows from the strong continuity of $U(t)$.
\end{proof}

\bibliographystyle{abbrv}
\bibliography{refs}

\end{document}